\documentclass[12pt]{article}
\usepackage{geometry}                
\usepackage{amsmath,amsthm,amsfonts,amssymb,bm,dsfont,mathtools}
\usepackage{graphicx}
\usepackage{xcolor}
\usepackage{hyperref}
\definecolor{burgundy}{rgb}{0.5, 0.0, 0.13}
\definecolor{linkblue}{RGB}{0,20,128}
\hypersetup{
  linkcolor  = linkblue,
  citecolor  = burgundy,
  urlcolor   = burgundy,
  colorlinks = true,
}

\numberwithin{equation}{section}
\usepackage{epstopdf}
\newcommand{\Es}{{\rm Es}}
\newcommand{\diam}{{\rm diam}}
\newcommand{\dist}{{\rm dist}}

\newcommand{\LE}{{\rm LE}}
\newtheorem{definition}{Definition}[section]
\newtheorem{theorem}[definition]{Theorem}

\newtheorem{lemma}[definition]{Lemma}
\newtheorem{cor}[definition]{Corollary}

\newtheorem{rem}[definition]{Remark}
\newtheorem{prop}[definition]{Proposition}

\newcommand{\KK}{{\cal K}}
\newcommand{\BZ}{{\mathbb{Z}}}
\newcommand{\BR}{{\mathbb{R}}}

\newcommand{\Zp}{{\mathbb{Z}^+}}
\newcommand{\Rp}{{\mathbb{R}^+}}
\newcommand{\Z}{{\mathbb{Z}}}

\newcommand{\R}{{\mathbb{R}}}
\newcommand{\Rt}{{\mathbb{R}^3}}
\newcommand{\PP}{{\mathbb{P}}}
\newcommand{\EE}{{\mathbb{E}}}
\newcommand{\cD}{{\cal D}}
\newcommand{\DD}{\mathbb{D}}
\newcommand{\cont}{{\rm Cont}}
\newcommand{\Dmo}{\mathbb{D}\setminus\{0\}}

\title{Sharp one-point estimates and Minkowski content for the scaling limit of three-dimensional loop-erased random walk}
\author{Saraí Hernández-Torres, Xinyi Li, and  Daisuke Shiraishi}
\date{December 2025}                                           

\begin{document}
\maketitle

\begin{abstract}
 In this work, we consider the scaling limit of loop-erased random walk (LERW) in three dimensions and prove that the limiting occupation measure is equivalent to its $\beta$-dimensional Minkowski content, where $\beta\in(1,5/3]$ is its Hausdorff dimension. In doing this we also establish the existence of the two-point function and provide some sharp estimates on one-point function and ball-hitting probabilities for 3D LERW in any scale, which is a considerable improvement of previous results.
\end{abstract}

\newpage
\tableofcontents

\newpage

\section{Introduction}\label{se:intro}

Understanding random self-repulsive curves is a challenging task in probability theory and statistical physics that has captured the attention of many researchers for more than half a century. In 1949, Flory proposed the model of a self-avoiding walk (SAW) as a toy model of polymers (see~\cite{Flo49}). In 1980, Lawler introduced in~\cite{Law1980} the model of loop-erased random walk (LERW), hoping to provide an alternative definition of SAW but gradually discovered that it is a model of significance in its own right.

The existence and characterization of the scaling limit of discrete random self-repulsive curves in $d=2,3$ are arguably among the most interesting questions in the model. Unlike the case in two dimensions, where Schramm-Loewner evolution (SLE) provides a precise description of the scaling limits of various critical lattice models, very little is known about the scaling limit of 3D critical models. A crucial reason for this is the lack of convenient tools, such as SLE, to describe the scaling limit intrinsically as a random curve.

In the 3D case, LERW was the first model to achieve a considerable breakthrough. In 2007, Kozma confirmed the existence of the scaling limit in \cite{Kozma}. Later, in \cite{Growth, Hausdorff, SS}, Shiraishi and Sapozhnikov-Shiraishi discovered various properties of this scaling limit. More precisely, it was confirmed that this scaling limit is almost surely a self-avoiding curve with a Hausdorff dimension $\beta$, where $\beta\in(1,5/3]$ is a certain constant. 

In considering the scaling limit of discrete random curves, it is natural to rescale the space by the inverse of the lattice mesh, parametrize the discrete curve in a way such that each lattice edge is traversed in the same amount of time and rescale the time in a way such that a smooth bounded domain is traversed in time of order 1. If the limit of the rescaled and parametrized random curve exists, we refer to it as the scaling limit in natural parametrization. Recently, the last two authors proved in \cite{Escape,Natural} that when rescaled properly, three-dimensional LERW converges as a parametrized curve to its scaling limit in natural parametrization and determined in \cite{Holder} the optimal H\"older exponent of the scaling limit.

With these results in mind, it is natural to ask if there is a more direct and canonical way to obtain the natural parametrization of the scaling limit. 
In the 2D case, Lawler and Viklund showed in \cite{LawVik,LawVikRadial} that the scaling limit of planar LERW in natural parametrization is given through the Minkowski content of ${\rm SLE}_2$, whose existence was confirmed in \cite{Law15}. It has also been shown (see \cite{LawShe,Ben}) that the natural parametrization of ${\rm SLE}_k$ (which satisfies a conformal Markov property) is unique up to multiplicative constants. However, conceivably it is difficult to establish such axiomatic characterization of natural parametrization in 3D due to lack of the rich conformal symmetry in 2D. 

In this work we are going to show that the Minkowski content of the scaling limit of 3D LERW exists and is equivalent to the limiting occupation measure obtained by the authors in \cite{Natural} from which the natural parametrization of the scaling limit is induced. In the course of proof, we also obtain sharp asymptotics of one-point function in any scale, and sharp estimates on the ball-hitting probabilities of LERW and its scaling limit, which is a big improvement of earlier results of the authors in \cite{Escape}.

\subsection{Main results} \label{subsec:main}

We now turn to the main results of this paper. Before stating them precisely, we need to introduce some notation first. Let $\mathbb{D}=\{ x \in \mathbb{R}^{3} \ | \ |x| < 1 \}$ be the open unit ball. Pick some $m\in\mathbb{R}^+$, 
and let $\mathbb{D}_m \coloneqq m^{-1}\mathbb{Z}^3\cap \mathbb{D}$ be the discretized unit ball on the rescaled integer lattice $m^{-1}\mathbb{Z}^3$.

Let $S_m$ be the simple random walk on $m^{-1}\mathbb{Z}^3$ started from $0$ and stopped at the first exit of $\mathbb{D}_m$. Define $\eta_m \coloneqq \LE(S_m)$ as its loop-erasure, which we will refer to as the LERW on $\mathbb{D}_m$. (We define the loop-erasure of a path below, in Definition~\ref{def:LE}.) Write $\KK_m$ for the trace of $\eta_m$ as a random compact subset of $\mathbb{R}^{3}$, and let $\KK$ stand for the weak limit of $\KK_{2^n}$ as $n\in\Zp$ tends to infinity, whose existence was confirmed in \cite{Kozma}. Let 
\begin{equation}\label{eq:betadef}
\beta\in \Big(1,5/3\Big]
\end{equation}
be the growth exponent of 3D LERW as well as the Hausdorff dimension of $\KK$ as defined in \cite{Growth} and \cite{Hausdorff}.

Let $ \hat{x} = (1/2, 0, 0)$ be a reference point, and let 
\begin{equation} \label{eq:scalingfactor}
    f_m \coloneqq m^{3} P \left( \hat{x}_m \in \eta_m \right)  
\end{equation}
be the \textbf{reference scaling factor}. Here, $\hat{x}_m$ stands for one of the nearest points to $\hat{x}$ among $m^{-1} \mathbb{Z}^{3}$.
From up-to-constant one-point estimates, we have $f_m \asymp m^{\beta}$ (see Subsection~\ref{section:oneandtwo} below).
Let 
\begin{equation} \label{eq:mudef}
  \mu_{m} \coloneqq (f_m)^{-1} \sum_{ x\in \mathbb{D}_m \cap \eta_m} \delta_{x}
\end{equation}
be the \textbf{occupation measure scaled by the one-point function}.

It was proved in~\cite[Theorem 1.1]{Natural} that as $m\in\Rp$ tends to infinity, 
\begin{equation}\label{eq:firstmu}
(\KK_{m}, \mu_{m})\xrightarrow{\rm w} (\KK, \mu)
\end{equation}
for some random measure $\mu$ supported on $\KK$, which is measurable with respect to $\KK$. 
We call $\mu$ the limiting occupation measure.

Define the following distance $\rho$ between two parametrized curves $\gamma^1:[0,T_1]\rightarrow\Rt$ and  $\gamma^2:[0,T_2]\rightarrow\Rt$ by
\begin{equation}\label{eq:rhodef}
\rho(\gamma^1,\gamma^2)=|T_2-T_1|+\sup_{0\le s\le 1}\left|\gamma^1(sT_1)-\gamma^2(sT_2)\right|.
\end{equation}
From now on, we regard the LERW $ \eta_m$ as a parametrized curve where each edge of $m^{-1}\mathbb{Z}^3$ is traversed in time $(f_m)^{-1}$ at a constant speed. Let $\eta$ be the random continuous curve obtained by parametrizing $\KK$ with the limiting occupation measure $\mu$, as described in~\cite[Subsection 9.1]{Natural}. It was also proved in \cite[Theorem 1.3]{Natural} that $\eta_{m}$ converges to $\eta$ in natural parametrization. In other words, as $m\in\Rp$ tends to infinity, 
\[
  \eta_{m}\xrightarrow{\rm w} \eta
\]
in the topology generated by the $\rho$-metric defined above in~\eqref{eq:rhodef}.

We now turn to Minkowski content. Roughly speaking, the Minkowski content of a set is defined through the limit of the renormalized volume of the enlargement of the set. Although not a measure per se, it is a natural way to define  ``fractal'' measures on  fractal sets. Given a Borel set $G\subset\Rt$ and $\delta\in (0,3]$, we write $\cont_{\delta}(G)$ for its $\delta$-dimensional Minkowski content (provided it exists); see Subsection \ref{sec:MC} for its precise definition. Let $\cD^o$ be the collection of ``nice'' dyadic boxes defined in \eqref{eq:dyadicdef}.

We are now ready to state our first result.
\begin{theorem}\label{mainthm}
There exists a universal constant $c_0>0$ such that for any dyadic box $V \in \cD^o$, 
$\cont_\beta(\KK \cap V)$ exists and equals to $c_0\mu(V)$ almost surely.
\end{theorem}
By a standard measure-theoretic procedure, see Proposition~\ref{prop:generalMC} below, we can generate a random non-atomic Borel measure $\nu$ out of the Minkowski content of $\KK$. 

Hence,  Theorem \ref{mainthm} implies the following corollary.
\begin{cor} \label{cor:MeasuresEqual}
With the same constant $c_0$ as in Theorem \ref{mainthm},
$$
\nu=c_0\mu \quad a.s. 
$$
\end{cor}

The Borel measure $\nu$ is supported on $\KK$ and defines an intrinsic parametrization. 
Given $x \in \KK$, denote by $\KK_x \subset \KK$ the trace of the curve between $0$ and $x$. 
For each $ t \in [0, \nu(\KK)/c_0] $, there uniquely exists $x_t \in \KK$ such that $t = \nu (\KK_{x_{t}})/c_0$, and then we define
\begin{equation} \label{eq:paramMinkowski}
  \eta^{\nu} (t) \coloneqq x_t, \qquad t \in [0, \nu(\KK)].
\end{equation}
The relevance of Corollary~\ref{cor:MeasuresEqual} is that the parametrization of $\eta$, which was obtained via the scaling limit of the occupation measure $\mu$ (which is a priori not necessarily scaling-invariant), is equal to an intrinsic parametrization of $\KK$, up to a universal constant factor. 
Moreover, since the measure $\nu$ is both scaling-invariant and rotation-invariant, then the {\it parametrized} limit curve $\eta$ is scaling-invariant and rotation-invariant as well (see Corollary~\ref{cor:maincor}).

In the course of the proof of Theorem \ref{mainthm}, we also obtain accurate asymptotics for any scale of the probability that LERW passes a certain point or a certain ball which we  refer to loosely as ``one-point functions''. 
We present these results here, since both of them are of independent interest. 

We start with one-point function estimates.
Given $x\in\Rt$, write $x_m\in m^{-1}\mathbb{Z}^3$ for the closest\footnote{Breaking ties arbitrarily, if necessary.} point on $m^{-1}\mathbb{Z}^3$. We now consider the probability that $\eta_m$ passes through $x_m$.

\begin{theorem}\label{thm:oneptnew}
There exists a function $g:\Dmo\to \mathbb{R}^+$ and universal constants $c,\delta > 0$  such that for all $m\in\Zp$ and $x  \in \mathbb{D} \setminus \{ 0 \}$, 
\begin{equation}\label{eq:oneptnew}
P \Big( x_{m} \in\eta_m \Big)= g(x) m^{\beta-3 } \Big[ 1 + O \Big( d_{x}^{-c} m^{-\delta} \Big) \Big] \qquad (\text{as } m \to \infty),
\end{equation}
where $d_{x} = \min \{ |x| , 1 - |x| \}$. Moreover, 
\begin{equation}
\mbox{ $g(x)$
is continuous in $\Dmo$ and rotation-invariant.}
\end{equation}
\end{theorem}

We now turn to the continuum and look at the probability of intersecting a ball. For $x \in \mathbb{R}^3$ and $r > 0$,
let $ B(x,r) \coloneqq \{ y \in \mathbb{R}^3 \, |  \, \vert x - y \vert \leq r \} $.
Elaborating on Theorem \ref{thm:oneptnew} and various coupling techniques developed so far, we are able to show the following asymptotics of the probability that the scaling limit $\KK$ hits a ball.

\begin{theorem}\label{thm:oneballnew}
There exist $c_1,c, \delta>0$ such that for any $x\in\Dmo$, for $r<\min\{|x|,1-|x|\}/2$,
\begin{equation}
P\Big( \KK \cap B(x,r) \neq \emptyset\Big)=c_1 g(x) r^{3-\beta} \Big[ 1+O\big(d_x^{-c}r^{\delta}\big)\Big],
\end{equation}
where $d_{x} = \min \{ |x| , 1 - |x| \}$.
\end{theorem}

Note that $g(x)$ here corresponds to $c(x)$ in~\cite[Theorem 1.1]{Escape} (which we rephrase in Proposition~\ref{prop:dyadic-onePoint}), where only dyadic scales were considered. Hence, $g(x)$ satisfies the following asymptotics (see~\cite[Equations (1.5), (1.6)]{Escape}):
\begin{equation} \label{eq:one-pointb0}
a_1|x|^{\beta-3} \leq g(x) \leq a_2|x|^{\beta-3} \mbox{ if } 0 <|x| < \frac12
\end{equation}
and
\begin{equation}\label{eq:one-pointb1}
a_1(1 - |x|)^{\beta-1}\leq g(x) \leq a_2(1 - |x|)^{\beta-1}\mbox{ if } \frac12\leq |x|<1
\end{equation}
for some $a_1,a_2\in\mathbb{R}^+$.
Thanks to Theorem~\ref{thm:oneballnew} and the scaling invariance of $\KK$, we can refine these asymptotics in the following corollary.

\begin{cor} \label{cor:asymptotics}
There exist constants  $b_1 > 0 $ and $\delta_1 > 0$ such that
  \begin{equation} \label{eq:asymp0}
    g (x) = b_1 \vert x \vert^{-(3 - \beta)} \Big[ 1 + O (\vert x \vert^{\delta_1}) \Big] \qquad \text{ as } x \rightarrow 0,
  \end{equation}
  and there exist $b_2 > 0 $ and $\delta_2 > 0$ satisfying
  \begin{equation} \label{eq:asymp1}
     g (x) = b_2 (1 - \vert x \vert)^{-(1 - \beta)} \Big[ 1 + O( (1 - \vert x \vert)^{\delta_2}) \Big] \qquad \text{ as } x \rightarrow 1.
  \end{equation}
\end{cor}

For the two-point function, we obtain the following estimates. 

  \begin{theorem} \label{thm:two-point}
  Let $ \Omega = \{ (z,w) \in \mathbb{D} \setminus \{ 0  \} \; \vert \; z \neq w \} $.
  There exists a function $g : \Omega \rightarrow \mathbb{R}^+ $ and universal constants $c, \delta > 0$ such that for all $m \in \mathbb{Z}^+$ and $(z, w) \in \Omega$
  \begin{equation} \label{eq:twopt-green}
    P \left( z_m, w_m \in \eta_m \right) = g(z,w) m^{2(\beta - 3)} \Big[  1 + o(1) \Big], \qquad \text{ as } m \to \infty.
  \end{equation}
  Moreover, for any $(z, w) \in \Omega$ and $r < d_{z,w} / 2$
  \begin{equation} \label{eq:twop-rate}
    P \left( \KK \cap B (z,r) \neq \emptyset, \KK \cap B (w,r) \neq \emptyset \right)
    = c_1^2 g (z,w) r^{2(3 - \beta)} \Big[  1 + O (d_{z,w}^{-c} r^{\delta}) \Big],
  \end{equation}
  where $d_{z,w} = \min \{  d_z, d_w, \vert z - w \vert \}$.
\end{theorem}

Theorem \ref{mainthm} and Theorem~\ref{thm:oneptnew} enhance~\cite[Theorem 1.3]{Natural}. Specifically, the refined one-point estimates allow us to substitute the reference scaling factor~\eqref{eq:scalingfactor} with an explicit scaling. 
We define
\begin{equation} \label{eq:explicitScaling}
  \overline{\mu}_{m} \coloneqq m^{-\beta} \sum_{x\in \mathbb{D}_m \cap \eta_m} \delta_{x}
\end{equation}
and introduce $\overline{\eta}_m$ as the LERW on $\mathbb{D}_m$ parametrized by $\overline{\mu}_m$ and $\overline{\eta}$ as $\KK$ parametrized by $(c_0g(\widehat x))^{-1}$-times its Minkowski content.
We now present the convergence of 3D LERW in its natural parametrization in full strength. 
\begin{theorem}\label{finalthm}
  As $m\to\infty$, $\overline{\eta}_m \xrightarrow{\rm w} \overline{\eta}$ in the topology generated by the $\rho$-metric.
\end{theorem}

The sharp one-point estimates also give us convergence of the infinite loop-erased random walk (ILERW) in its natural parametrization.

Let $ \mathcal{C} $ be the space of parametrized curves $\lambda : [0, \infty) \to \mathbb{R}^3 $ satisfying $\lambda (0) = 0$ and  $\lim_{t \to \infty} \diam (\lambda [0,t]) = \infty $.
Consider the metric on $\mathcal{C}$ given by
\[
  \chi (\lambda_1, \lambda_2) \coloneqq \sum_{k = 1}^{\infty} 2^{-k} \max_{0 \leq t  \leq k} \, \min \left\lbrace | \lambda_1 (t) - \lambda_2 (t)|, 1 \right\rbrace,  
\]
for $\lambda_1, \lambda_2 \in \mathcal{C}$.
The space $(\mathcal{C}, \chi)$ is a complete and separable metric space (see Subsection 2.4 of~\cite{KS}). 

Let $S_{m}$ be a simple random walk on $ m^{-1} \mathbb{Z}^3 $ started from $0$. The path
\[ 
  \eta^{\infty}_m \coloneqq \LE (S_m [0, \infty)  )
\] 
defines the \emph{infinite loop-erased random walk} (ILERW) on $m^{-1}  \mathbb{Z}^3$. Since the simple random walk on $\mathbb{Z}^3$ is transient almost surely, then $\eta^{\infty}_m$ is well-defined.
We also write $\eta^{\infty}_m$ for the curve obtained by linearly interpolating the path $\eta^{\infty}_m$, and consider a parametrization $\overline{\eta}^{\infty}_m$ defined by
\[
  \overline{\eta}^{\infty}_m (t) \coloneqq \eta^{\infty}_m ( m^{\beta} t ), \qquad t \geq 0.
\]
The continuous curve $\overline{\eta}^{\infty}_m$ defines a random element of $\mathcal{C}$.

\begin{theorem} \label{thm:ILERW}
  There exists $\eta^{\infty} \in \mathcal{C}$ such that, 
  as $m \to \infty$,  $\overline{\eta}^{\infty}_m \xrightarrow{\rm w} \eta^{\infty}$ with respect to the topology generated by the metric $\chi$.
\end{theorem}

\begin{rem}
1) Using the same techniques employed in proving Theorem~\ref{thm:oneptnew}, one can show the existence of the one-point function associated with $\eta^{\infty}$, the scaling limit of the ILERW. Specifically, letting  $g_{\infty} :=b_1 |x|^{-(3-\beta)}$ (recall the definition 
of $b_1$ from Corollary \ref{cor:asymptotics}), using an asymptotic relation between LERW and ILERW (see e.g.\ Lemma 3.1 of \cite{Natural}), there exist $c,\delta>0$ such that for any $x\in\mathbb{D}\setminus\{0\}$,
\begin{equation}
P\Big(x_m\in \eta_m^\infty\Big)=g_\infty(x) m^{\beta-3 } \Big[ 1 + O \Big( |x|^{-c} m^{-\delta} \Big) \Big] \qquad (\text{as } m \to \infty).
\end{equation}
Moreover, for all $r\in(0,|x|/2)$,
\[
  P \left( \eta^{\infty} \cap B (x, r) \neq \emptyset \right) = c_1 g_{\infty}(x) r^{3 - \beta} \Big[  1 + o_r(1) \Big],
\]
where $c_1$ is defined in Theorem \ref{thm:oneballnew}. In a similar fashion, we can also define the two-point functions $g_\infty(\cdot,\cdot)$. We omit the details.


\smallskip

\noindent 2) Although open at the moment, we conjecture that $g(\cdot)$ is differentiable and $|g'(\cdot)|$ satisfies certain asymptotics as $x\to0$ or $x\to1$.

\smallskip

\noindent 3) Under certain regimes, by performing an analysis similar to those in this work, it is possible to relate the asymptotics of $g(\cdot,\cdot)$ and $g_\infty(\cdot,\cdot)$ to those of $g(\cdot)$ or $g_\infty(\cdot)$. For example, as $x\to0$ or $x\to1$, for any $y\in\mathbb{D}\setminus\{0\}$, one has
\begin{equation}
    \frac{g(x,y)}{g(x)}\to g(y).
\end{equation}   
This is reminiscent of the operator product expansion (OPE) in 2D conformal field theory.
\end{rem}

\begin{rem}1)
  The existence of a scaling limit for the uniform spanning tree on $\Z^3$ was established in~\cite{3dUST}. 
  The paper relies on the scaling limit of the loop-erased random walk in~\cite{Natural} and the one-point function estimates for the loop-erased random walk in~\cite{Escape}. 
  Building upon the results from these papers,~\cite{3dUST} shows the existence of sub-sequential scaling limits, convergence over a dyadic scaling, and properties of the limit tree. 

  As noted in~\cite[Remark 1.2]{3dUST}, the restrictions of~\cite{Natural} and~\cite{Escape} to a dyadic scaling impose the same condition over the results in~\cite{3dUST}. 
  With the introduction of Theorems~\ref{thm:oneptnew},~\ref{finalthm}, and~\ref{thm:ILERW}, the convergence of the scaling limit of the UST on $\mathbb{Z}^3$ now holds for arbitrary scaling. This convergence is established with respect to the Gromov-Hausdorff-Prohorov topology, incorporating a locally uniform topology for the embedding. For a detailed presentation of this topology, see~\cite{3dUST}.
\end{rem}

\subsection{Proof overview}

We now briefly comment on the proof. To prove the existence of Minkowski content, we follow the relatively standard route (see e.g.~\cite{Law15},~\cite{LawRaz} or~\cite{HLLS}) in which the crux is a strong two-point estimate (see Proposition \ref{prop:crux2pt} for the exact statement).  
 As we do not yet possess sufficient knowledge of the scaling limit, it is relatively hard to obtain such estimates from the continuum side directly. However, a continuity result on hitting probability for balls (see Proposition \ref{prop:continuity}) allows us to work in the discrete instead. We then apply the asymptotic independence of LERW in the two-point case (see Proposition \ref{prop:decoup} for precise statements). The asymptotic independence was originally obtained in \cite{Natural} as the core property to ensure the convergence in natural parametrization. Intuitively speaking, this property allows us to ``decouple'' global behavior and local behavior, and reduce two-point estimates to a strong one-point estimate, which we will comment shortly. Similar arguments lead to the identification (up to a multiplicative constant) of the measure induced by Minkowski content and the limiting occupation measure.

We now comment on the one-point estimates.  Recall that in~\cite{Escape}, for obtaining the strong one-point function asymptotics, the authors compared  the non-intersection probability of independent LERW and SRW started at the origin, with the non-intersection probability of independent LERW and SRW with separated starting points; the gap between these starting points allowed the comparison of the LERW with its scaling limit and the non-intersection probability with respect to an independent Brownian motion. 
To extend the one-point  estimate (i.e., Theorem \ref{thm:oneptnew}) beyond dyadic scales, 
we examine scales of the form $m = 2^{n + r}$ with $n \in \mathbb{N}_+$ and $r \in [0,1)$.
Given a fixed $x \in \mathbb{D} \setminus \{ 0 \} $, we consider a function $f : [0,1] \rightarrow \mathbb{R}^+$ satisfying $ P ( x_{2^m} \in \eta_{2^{m}} ) \simeq f(r) 2^{-(3 - \beta)n}  $. In fact, the function $f$ satisfies a certain functional equation (stated in Proposition~\ref{prop:funceq}) which implies that $f (r) = c 2^{-(3 - \beta)r}$ for some $c > 0$. Therefore $P ( x_{2^m} \in \eta_{2^{m}} ) \simeq c 2^{-(3 - \beta)m} $ for any $m \in \mathbb{R}^+$ and yields Theorem \ref{thm:oneptnew}. The key part property of the functional equation is a local logarithmic additivity property, which is proved in Proposition~\ref{prop:funceqss} with the techniques in~\cite{Escape} mentioned above.
 
\subsection{Organization of this work}

This work is organized as follows. In Section \ref{se:notations} we will introduce various notation and basic properties of LERW that we will be using repeatedly.  
In  Section~\ref{se:continuity} we prove  the continuity property of ball-hitting probability of Proposition  \ref{prop:continuity}.
This proposition will allow us to pass results from the discrete to the continuum. 
In Section~\ref{se:onept} we begin with our analysis on one-point function estimates. We then give the proof of Theorems~\ref{thm:oneptnew} and~\ref{thm:oneballnew}. These two theorems are the technical basis for our main results. By the end of Section~\ref{se:onept}, we prove the asymptotics for the one-point function of Corollary~\ref{cor:asymptotics}. Moving on to Section~\ref{se:twopoint}, we establish an $L^2$ estimate in Proposition~\ref{prop:3.0}, which follows from the crucial two-point estimate in Proposition~\ref{prop:crux2pt}. Then we prove Theorem~\ref{thm:two-point}.  
Finally, in Section \ref{se:proof} we prove the main result, Theorem \ref{mainthm}.


We end this section by mentioning our convention for constants in this work. Unless otherwise specified, in this work we always consider $d=3$. We use $c,C,c',\ldots$ to denote positive constants which may change from line
to line and use $c$ with subscripts, i.e., $c_1,c_2,\ldots$ to denote constants that stay fixed. If a constant is dependent upon some other quantity, this dependence will be made explicit. For example, if $C$ depends on $\delta$, we write $C(\delta)$. We write $f_{n} \simeq g_{n}$ to indicate that $f_n = g_n [1 + O (2^{-n\delta})]$ for  $\delta > 0$ as $n \to \infty$.

\section{Notation and some useful facts}\label{se:notations}
In this section, we will fix some notation and conventions that we will stick to throughout this paper. We will also recall some useful facts regarding loop-erased random walk and its scaling limits, in particular in three dimensions.

In particular, in Subsection \ref{se:scalinglimit}, we will discuss the existence of the scaling limit of the trace of 3D LERW, as originally proved in \cite{Kozma}. We will also state a version of this convergence, which is crucial to this work and whose proof was already present within \cite{Kozma}.
\subsection{General notation}
Let $\BZ^3$ stand for the three-dimensional lattice and for any $m>0$, let $m^{-1}\BZ^3$ stand for the rescaled lattice of mesh size $m^{-1}$. For any open $D\subseteq \BR^3$, we write $D_m = D\cap m^{-1}\BZ^3$ for the discretization of $D$ on $m^{-1}\BZ^3$. 

We call $\lambda = [\lambda(0),\lambda(1),\ldots ,\lambda(n)]$, a sequence of points in $m^{-1}\BZ^3$ for some $m>0$ a path, if $|\lambda(j+1) - \lambda(j)| = m^{-1}$ for all $j= 0,\ldots,n-1$. Let ${\rm len}(\lambda) = n$ denote the (graph) length of $\lambda$. We call $\lambda$ a \emph{self-avoiding path (SAP)} if $\lambda(i)\neq \lambda(j)$ for all $i\neq j$.

\begin{definition}\label{def:LE}
Given a path $\lambda =[\lambda(0),\lambda(1),\ldots,\lambda(m)]$, we define its chronological loop-erasure $\LE(\lambda)$ as
follows. Let
$$
s_0 \coloneqq \max\{t \geq 0 \, | \, \lambda(t) = \lambda(0)\},
$$
and for $i\geq1$, let
$$
s_i \coloneqq \max\{t \geq 0 \, | \, \lambda(t) = \lambda(s_{i-1}+1)\}.
$$
We write $n = \min\{i \, |  \, s_i = m\}$. Then,
$$\LE(\lambda) \coloneqq [\lambda(s_0), \lambda(s_1),\ldots , \lambda(s_n)].$$
\end{definition}
It is easy to see that $\LE(\lambda)$ is an SAP.

We call an infinite sequence of points $\lambda=[\lambda(0),\lambda(1),\ldots]$ in $m^{-1}\BZ^3$ an \emph{infinite path}, if $|\lambda(j+1) - \lambda(j)| = m^{-1}$ for all $j= 0,1,\ldots$. We call $\lambda$ \emph{transient} if for any $a\in m^{-1}\BZ^3$, $|\{k: \lambda(k)=a\}|<\infty$. For a transient infinite path, we can define its loop erasure, denoted as $\LE(\lambda)$, in a manner similar to Definition~\ref{def:LE}.

We write $|\cdot|$ for the Euclidean distance in $\Rt$. Given $B_1,\ldots,B_n$ a sequence of Borel subsets of $\Rt$, let
\begin{equation}\label{eq:distdef}
\dist(B_1,\ldots,B_n) \coloneqq \inf\{ |x-y|: x\in B_i, y\in B_j, i\neq j\}.
\end{equation}
If $B_1=\{b\}$ consists of only one point, with little abuse of notation we will write $\dist(b,B_2)$ for short.

Given a Borel set $G\subset\Rt$, we write ${\rm Vol}(G)$ for its volume (Lebesgue measure in $\Rt$) and write
\begin{equation}
B(G,r)\coloneqq \{x\in\Rt \, |  \, \dist(x,G)\leq r\}
\end{equation}
for the $r$-neighborhood of $G$ for some $r>0$.

\subsection{Metric spaces}\label{metric}
For a closed $D\subset\Rt$, 
we let $\big( {\cal H} (D ), d_{\text{Haus}} \big)$ be the space of all non-empty compact subsets of $D$ endowed with the \emph{Hausdorff metric} 
\begin{equation}\label{Haus}
d_{\text{Haus}} (A, B) \coloneqq \max \Big\{ \sup_{a \in A} \inf_{b \in B} |a-b|,  \  \sup_{b \in B} \inf_{a \in A} |a-b| \Big\} \  \text{ for } A, B \in {\cal H} (D).
\end{equation}
It is well known that $\big( {\cal H} (D), d_{\text{Haus}} \big)$ is a compact, complete metric space (see \cite{Henr} for example).

We set $\big( {\cal C} (D ), \rho \big)$ for the space of continuous curves $\lambda : [0, t_{\lambda}] \to D $ with the time duration $t_{\lambda} \ge 0$, where the metric $\rho$ is defined by 
\begin{equation}\label{rho-metric}
\rho (\lambda_{1}, \lambda_{2} ) \coloneqq |t_{1} - t_{2} | + \sup_{0 \le s \le 1} \big| \lambda_{1} ( s t_{1} ) - \lambda_{2} ( s t_{2} ) \big| \ \text{ for } \lambda_{1}, \lambda_{2} \in {\cal C} (D ).
\end{equation}
It is well-known that $\big( {\cal C} (D ), \rho \big)$ is a separable metric space (see Subsection 2.4 of \cite{KS}).

\subsection{Minkowski content}\label{sec:MC}
In this subsection, we give a brief introduction to the notion of Minkowski content and discuss how to induce a measure out of it. For generality, we will write down the results in any dimension $d$.

Let $A\subset \R^d$ be a bounded Borel set. For some $\delta\in(0,d]$ and $r>0$, we write
\[ 
  \cont_{\delta,r}(A)\coloneqq r^{-d+\delta} {\rm Vol}\big(B(A,r)\big)
\]
for the $\delta$-dimensional $r$-content of $A$, and write
\[ 
  \cont_\delta(A)\coloneqq\lim_{r\to 0}  \cont_{\delta,r}(A)
\]
for its $\delta$-dimensional Minkowski content, {\it provided the limit exists}. Note that for each $A$, there exists at most one $\delta$ such that its $\delta$-dimensional content is finite and positive. 
The definition of the Minkowski content depends on the volume. Then, when the corresponding limit exists, the Minkowski content enjoys a scaling property and is invariant under the isometries of $\R^d$.

Suppose $A$ is a random compact subset of $\BR^d$ governed by some probability measure $\mu$ with Hausdorff dimension $\delta\in (0,d)$. In order to construct a fractal measure supported on $A$, we need to show that for a wide class of sets $V$, $\cont_\delta(A)$ exists. 

To this end, we employ a classical method, which involves proving the existence of one-point and two-point Green's functions with some convergence rate. The respective conditions and claims are summarized in the following proposition (quoted from~\cite[Proposition 4.8]{HLLS}, with some adaptations, see Remark \ref{rem:modiMC}).

We write $\mu [ X ]$ for the integral of $X$ with respect to $\mu$.

\begin{prop}\label{prop:generalMC}
Let $A$ be  a random compact subset of $\BR^d$ governed by some probability measure $\mu$.   
Fix $0 < \delta < d $ and set $ \eta = d - \delta $. 
Suppose $U,V$ are  bounded Borel subsets of $\R^d$
with $V \subset {\rm int}(U)$ and
 such that for some $\varepsilon > 0$,
\begin{equation}\label{eq:boundarycond}
\cont_{d- \varepsilon}(\partial V)  = 0 .
\end{equation} 
Let 
$$
J_s(z) =   2^{\eta s} 1_{\{\dist(z,A)  \leq 2^{-s} \}}
\mbox{ and }
  J_{s,V} = \int_V J_s(z) \, dz.$$
Suppose   the following holds for $z \neq w \in U$.
\begin{itemize}
\item The limits
\begin{equation}  \label{eq:mereexistence}
 G(z) = \lim_{s \rightarrow \infty}\mu\left[J_s(z) 
\right] \mbox{ and }  G(z,w) = \lim_{s \rightarrow \infty}
    \mu\left[J_s(z) \, J_s(w)\right] 
\end{equation}
exist and are finite.  Moreover, 
$G$ is uniformly bounded on $U$. 
\item  There exist  $c,b, \rho_0,u >0$ such that
 for $ 0 \leq \rho \leq \rho_0$,
\begin{equation}
 \label{eq:speed}
 \begin{split}
 &G(z) \leq c; \;\; G(z,w) \leq  c \, |z-w|^{-\eta};
   \;\;   \left| \mu[J_s(z)] - G(z)\right| \leq c \, e^{-us^b};\\
 &\left|\mu[J_s(z) \, J_{s+\rho}(w)] -
 G(z,w)\right| \leq  c \, |z-w|^{-\eta} \, e^{-us^b}.
 \end{split}
  \end{equation}
 \end{itemize}
 Then the finite limit
$ J_V =  \lim_{s \rightarrow \infty} J_{s,V} $
exists both a.s.\ and in $L^2$. In the meanwhile,\begin{equation}\label{eq:boundaryV}
\cont_{\delta}(\partial V\cap A)=0\mbox{ a.s.}
\end{equation}
 Moreover, one has
\begin{equation}\label{eq:JV}
J_V = \cont_\delta(V \cap A), \mbox{ a.s.},
\end{equation}
\begin{equation}    \mu\left[J_V \right] =  \lim_{s \rightarrow \infty}
   \mu\left[J_{s,V} \right] = \int_V G(z) \, dz, \mbox{ and}
\end{equation}
\begin{equation}    \mu\left[J_V^2 \right] =  \lim_{s \rightarrow \infty}
   \mu\left[J_{s,V}^2 \right] = \int_V  \int_V G(z,w) \, dz \, dw.
\end{equation}
  \end{prop}
  
\begin{rem}\label{rem:modiMC}
Compared to~\cite[Proposition 4.8]{HLLS}, in the proposition above, we have weakened the conditions on the convergence rates of one-point and two-point functions in \eqref{eq:speed}. More precisely, the version in \cite{HLLS} is equivalent to taking $b=1$ here. Note that this modification does not change the argument. More precisely, the RHS of (4.41) and (4.42), ibid., stay summable even when $e^{-us}$ is replaced by $e^{-us^b}$ for some $b>0$, which implies the RHS of (4.43) and (4.44), ibid., still converges to $0$.  We have made this modification to suit the convergence rate obtained in \cite{Natural}, as stated in Proposition \ref{prop:decoup}, in particular the error terms in \eqref{eq:ppbb} and \eqref{eq:bpbb}.
\end{rem}

We now discuss how to induce a (random) Borel measure supported on $A$ via its Minkowski content.  We call $V\subset\R^d$ a dyadic box of scale $n$, if 
\begin{equation}\label{eq:dyadicdef}
V=\left(\frac{k_1}{2^n},\frac{k_1+1}{2^n}\right]\times\cdots\times\left(\frac{k_d}{2^n},\frac{k_d+1}{2^n}\right], \qquad k_1,\ldots,k_d\in\Z
\end{equation}
 for some $n\in\Zp$. 
We write $\cD_n$ for the collection of dyadic boxes $V$ of scale $n$ and write \begin{equation}\label{eq:cddef}
\cD=\cup_{n\in \Z^+} \cD_n
\end{equation}
for the collection of all dyadic boxes.

With the application to the model we consider in this work (namely the scaling limit of 3D LERW in a ball) in mind, we now restrict our attention to a specific sub-collection of dyadic boxes. The same argument works for other boundary conditions or models with little modification.
 We write $\cD_n^o$ for all $V\in \cD_n$ such that $V\subset\Dmo$ and 
\begin{equation}
\dist(0,\partial \DD, V) \geq 2^{-n}
\end{equation}
(see \eqref{eq:distdef} for the definition of $\dist(\cdot)$) and write 
\begin{equation}\label{eq:dnodef}
\cD^o=\cup_{n\in\Z^+} \cD_n^o.
\end{equation} By the same argument as in~\cite[Appendix A]{HLLS}, one arrives at the following proposition.
\begin{prop}
If for some random closed subset $A$ of $\mathbb{D}$ and some $\delta\in(0,d)$, the conditions in Proposition \ref{prop:generalMC} simultaneously hold for all $V\in \cD^o$, then almost surely, the $\nu$ defined as follows:
\[
\nu(V) \coloneqq \cont_{\delta}(A\cap V) \mbox{ \  for all \ }V\in  \cD^o
\]
can be extended to a random non-atomic Borel measure supported on $A$.
\end{prop}

\begin{rem} We now briefly explain why the restriction to the collection of dyadic boxes $\cD^o$ is necessary. Let $A$ be a random closed subset of $\mathbb{D}$.  Note that for some open set $U$
\[
\cont_\delta(A\cap U) =\cont_\delta (A\cap \overline{U}),
\]
(assuming both sides exist). Hence one can not expect  that, a priori, the outer measure defined by
\[
\nu(G)\coloneqq \cont_\delta(A\cap G), \qquad G \subset \mathbb{R}^d
\]
indeed induces a measure. However, for random fractals such as the scaling limit of LERW, one can check that conditions of Proposition \ref{prop:generalMC} hold for boxes that are not too close to the boundary of the domain or the starting point of the curve. This justifies our choice of $\cD^o$.
\end{rem}

\subsection{Weak convergence of probability measures}\label{PROKH}
In this subsection, we let ${\cal M} (\overline{\mathbb{D}} )$ be the space of all finite measures on $\overline{\mathbb{D}} $ equipped with the topology of the weak convergence. The space ${\cal M} (\overline{\mathbb{D}} )$ is complete, metrizable and separable (see Section 6, in particular~\cite[Theorem 6.8]{Bil}).

We now briefly recall basic facts on the weak convergence of probability measures here. See \cite[Chapter 3]{EK} for details and proofs.

Let $(M, d)$  be a metric space with its Borel sigma algebra ${\cal B} (M)$. We denote the space of all probability measures on $ \big( M,  {\cal B} (M) \big) $ by ${\cal P} (M)$, where the space ${\cal P} (M)$ is equipped with the topology of weak convergence. For a subset $A \subset M$ and $\varepsilon > 0$, we write $A_{\varepsilon} = \{  x \in M \ | \ \exists y \in A \text{ such that } d(x,y) < \varepsilon \}$ for the $\varepsilon$-neighborhood of $A$. The \emph{Prokhorov metric} $\pi : {\cal P} (M)^{2} \to [0,1]$ is defined by 
\begin{equation}\label{prokhorov}
    \pi (\mu, \nu) \coloneqq 
    \inf \Big\{ \varepsilon > 0 \ \Big| 
    \begin{array}{c} 
    \text{for all } A \in {\cal B} (M)  \\
    \ \mu (A) \le \nu (A_{\varepsilon} ) + \varepsilon \text{ and } \nu (A) \le \mu (A_{\varepsilon} ) + \varepsilon  
    \end{array}
    \Big\}.
\end{equation}

It is well known that if $(M, d)$ is separable, convergence of measures in the Prokhorov metric is equivalent to weak convergence of measures. Therefore, $\pi$ is a metrization of the topology of weak convergence on ${\cal P} (M)$. It is also well known that if $(M, d)$ is a compact metric space, the metric space $\big( {\cal P} (M), \pi \big) $ is also compact. If $(M,d)$ is a complete and separable metric space, then the metric space $\big( {\cal P} (M), \pi \big) $ is also complete.

We now turn to random sets. To characterize the speed of weak convergence of random closed sets with respect to the $d_{\text{Haus}}$ metric, we introduce the so-called  L\'evy metric as follows. For $X,Y$  two random closed subset of $\mathbb{D}$, writing $T_X(K)=P(X\cap K\neq\emptyset)$ for closed $K\subset\mathbb{D}$, we define their distance under  \emph{L\'evy metric} as
\begin{equation}\label{eq:Levy}
    d_{\text{L\'evy}}(X,Y) \coloneqq
    \inf \left\{\varepsilon >0  \ \Big| 
    \begin{array}{c}
    \mbox{for all closed }K\subset\mathbb{D} \\
    \ T_X(K)\leq T_{B(Y,\varepsilon)}(K)+\varepsilon, \\  T_Y(K)\leq T_{B(X,\varepsilon)}(K)+\varepsilon
    \end{array}
    \right\}.
\end{equation}
It is well-known that the convergence under L\'evy metric is equivalent to the weak convergence of random closed sets with respect to the $d_{\text{Haus}}$-metric, see \cite{Molchanov}.

\subsection{The scaling limit of 3D LERW}\label{se:scalinglimit}
In this subsection, we will briefly review known results on the existence of the scaling limit of the three-dimensional loop-erased random walk, and give a slightly improved version of Kozma's result tailored to our needs in this work.

Recall that the metric space $\big( {\cal H} (\overline{\mathbb{D}} ), d_{\text{Haus}} \big)$ was defined as in \eqref{Haus} and the L\'evy metric was defined in \eqref{eq:Levy}.  Recall that for $m>0$, $\eta_m$ stands for the LERW in $\mathbb{D}_m$. With slight abuse of notation (and truncating the last step if needed), we can regard the LERW $\eta_{m}$ as a random element of ${\cal H} (\overline{\mathbb{D}} )$. 
Gady Kozma made the first breakthrough on the existence of the scaling limit of 3D LERW. In \cite{Kozma}, it was proved that the trace of 3D LERW converges as a random set along dyadic scales, with a polynomial speed in the L\'evy metric. More precisely, it was proved (see Theorems 5 and 6, ibid.) that for $n\in\Zp$, there exists $\KK$, a random closed subset of $\mathbb{D}$, and $0<c,C<\infty$ such that
\begin{equation}\label{eq:convdyadic}
d_{\text{L\'evy}}(\eta_{2^n},\KK)\leq C2^{-cn}.
\end{equation}
In the same work, it was also proved (see Subsection 6.1, ibid.) that the $\KK$ is rotation-invariant. Moreover, the scaling limit of 3D LERW is scaling-invariant in the following sense:
for $r>0$ write
\begin{equation} \label{eq:scaleInvariance}
  \KK^r \coloneqq r\KK
\end{equation}
for the $r$-blowup of $\KK$ and let $D_{r,2^n}=2^{-n}\mathbb{Z}^3\cap r\mathbb{D}$ be the discretization of $r\mathbb{D}$ and write $\eta^{r}_{2^n}$ for the LERW in $D_{r,2^n}$ started from the origin. Then 
there exist constants $c(r),C(r)\in(0,\infty)$, such that
\begin{equation}\label{eq:convrdya}
d_{\text{L\'evy}}(\eta^{r}_{2^n},\KK^r)\leq C2^{-cn}.
\end{equation}
In other words, $\KK^r$ is the scaling limit of LERW in the domain $r\mathbb{D}$ along dyadic scales.

It is noteworthy that the scaling invariance of $\KK$ actually implies that the convergence in \eqref{eq:convdyadic} follows not only along dyadic scales, but also as $m\to\infty$ continuously. See~\cite[Subsection 4.1]{Natural} for a detailed explanation of the proof of the following proposition.

\begin{prop}\label{prop:Hanyscale} 
There exists $c,C\in (0,\infty)$ such that as $m\to \infty$
\begin{equation}
d_{\text{L\'evy}}(\eta_{m},\KK)\leq Cm^{-c}.
\end{equation}
\end{prop}

Given the equivalence of weak convergence and L\'evy metric (as discussed at the end of Subsection \ref{PROKH}), one can view 
Proposition~\ref{prop:Hanyscale} as a quantitative statement of the weak convergence of LERW towards its scaling limit. Consequently, the weak convergence is a corollary of Proposition~\ref{prop:Hanyscale}.

\begin{cor} \label{cor:wCONV}
  For $m\in\R^+$,
  \[
    \eta_m\overset{\rm w}{\to}\KK\qquad\mbox{ as }\qquad m\to\infty
  \]
  in the topology generated by the Hausdorff metric.
\end{cor}

  Some topological properties of ${\cal K}$ are studied in \cite{SS}. We recall that $\gamma \in {\cal H} (\overline{\mathbb{D}} ) $ is a simple curve if $\gamma$ is homeomorphic to the interval $[0, 1]$. For a simple curve $\gamma \in {\cal H} (\overline{\mathbb{D}} )$, we write $\gamma^{s}$ and $\gamma^{e}$ for its end points. Let 
  \begin{equation}\label{SIMPLE}
  \Gamma \coloneqq \big\{ \gamma \in {\cal H} (\overline{\mathbb{D}} ) \ | \  \gamma \text{ is a simple curve satisfying }  \gamma^{s} = 0 \text{ and } \gamma \cap \partial \mathbb{D} = \{  \gamma^{e} \}  \big\}.
  \end{equation}
 Theorem 1.2 of \cite{SS} shows that ${\cal K} \in \Gamma$ almost surely. In \cite{Hausdorff}, it is proved that the Hausdorff dimension of ${\cal K}$ is equal to $\beta$ almost surely.
  
 In \cite{Natural}, the last two authors improve the convergence result by showing that the (renormalized) occupation measure of LERW has a scaling limit and rescaled LERW converges as a naturally parametrized curve to the curve given through parametrizing $\KK$ by the limiting occupation measure, as discussed in Subsection~\ref{subsec:main}. 
 More precisely, we have the following theorem. Recall that
\[
  \mu_{m}=  (f_m)^{-1} \sum_{x\in \mathbb{D}_m} \delta_{x\in \eta_m}
\]
is the occupation measure of LERW scaled by the one-point function.
\begin{theorem}\label{thm:weakconv}
As $m\in\Rp$ tends to infinity, 
\begin{equation}\label{eq:weakconv}
(\KK_{m},\mu_{m})\xrightarrow{\rm w} (\KK, \mu)
\end{equation}
for some random $\mu$ supported on $\KK$ which is measurable with respect to $\KK$, in the product topology of Hausdorff metric for sets and weak convergence topology for measures.
\end{theorem}
We call $\mu$ the limiting occupation measure. Let $\eta$ be the random continuous curve obtained through parametrizing $\KK$ by $\mu$. 

The one-point estimates over dyadic scales along dyadic scales in~\cite{Escape}, allow taking the scaling factor $f_{2^n}$ in \eqref{eq:scalingfactor} as a multiple of $ 2^{\beta n}$, for $n\in\mathbb{Z}^+$. Recall that we write
\[
  \overline{\mu}_{2^n}= 2^{-\beta n} \sum_{x\in \mathbb{D}_{2^n}} \delta_{x\in \eta_{2^n}}
\]
and $\overline{\eta}_{2^n}$ is the LERW on $\mathbb{D}_{2^n}$ parametrized by $\overline{\mu}_{2^n}$ (compare~\eqref{eq:explicitScaling}).
Therefore, the results in~\cite{Escape, Natural} imply that $\bar{\eta}_{2^{n}}$ converges weakly to some $\overline{\eta}$ in natural parametrization. In other words, as $n\in\Zp$ tends to infinity, 
\[
    \overline{\eta}_{2^n}\xrightarrow{\rm w} \overline{\eta}
\]
in the topology generated by $\rho$-metric. 

\subsection{Up-to-constant  one- and two-point estimates} \label{section:oneandtwo}
In this subsection, we present various up-to-constant one- and two-point estimates on the hitting probabilities for three-dimensional loop-erased random walk as well as its scaling limit. These estimates will play an important role in the next section. 

We start with one-point estimates. Given $x\in\Dmo$, write $d_x=\min\{|x|,1-|x|\}$ and let
\begin{equation}\label{eq:Cxdef}
C_x=\begin{cases} 
d_x^{-3+\beta} &\mbox{ if }|x|\leq 1/2; \\
d_x^{\beta-1}&\mbox{ if } |x|>1/2.
\end{cases}
\end{equation}

The following estimates follow from~\cite[Equation (5.13)]{Natural} and~\cite[Corollary 7.4]{Natural}.

\begin{prop}\label{prop:uptocb1}
For $0<r<d_x/100$, there exists some $M_0(x,r)$ such that for all $m\geq M_0$,
\begin{equation}
P\Big(\eta_m \cap B(x,r) \neq \emptyset\Big)\asymp C_x  r^{3-\beta}. \label{eq:oneballd}
\end{equation}
Equivalently, with the same assumption as above,
\begin{equation}
P\Big(\KK \cap B(x,r) \neq \emptyset\Big)\asymp C_x  r^{3-\beta}.
\end{equation}
\end{prop}

A corollary of this is the following conclusion that $\KK$ leaves no content on the surface of any dyadic box.
\begin{cor}\label{cor:plane}
Let $L\subset\Rt$ be a plane. Then ${\cont}_\beta(\KK \cap L)=0$ a.s.
\end{cor}

In \cite{Escape}, the last two authors obtain the following sharp one-point function asymptotics (see Theorem 1.1, ibid.) along dyadic scales. Recall that $\eta^m$ stand for the LERW on $\mathbb{D}_m$ as defined in Section \ref{se:intro}.

\begin{prop} \label{prop:dyadic-onePoint}
With the same assumptions as Proposition \ref{prop:uptocb1}, there exist universal constants $c > 0$, $\delta > 0$ and a constant $c(x) > 0$ such that for all $n\in\Zp$, 
\begin{equation}\label{eq:sharpone}
P\Big(x_{2^n} \in \eta_{2^n} \Big)= c(x) 2^{(\beta-3)n}\Big[ 1+O\big(d_x^{-c}2^{-\delta n}\big)\Big], \qquad \text{ as } n \to \infty.
\end{equation}
\end{prop}

A corollary of this result is the tightness of the total length of LERW. Recall $\mu$ is the (rescaled) occupation measure defined in \eqref{eq:mudef}. Then, for $m>0$,
\begin{equation}
M_{m} \coloneqq \mu_{m}(\DD)
\end{equation}
gives the the total length of $\eta_m$.
\begin{cor}[{\cite[Corollary 1.3]{Escape}}]\label{cor:tight}
The family of random variables $2^{-\beta n} M_{2^n}$, $n\in\Zp$ are tight. 
\end{cor}

We now turn to two-point estimates. Recall the definition of $\cD^o$ in~\eqref{eq:dnodef}. Let $V\in\cD^o$. Write $x_V$ for the center of $V$ and
$C_V=C_{x_V}$ where $C_{x_V}$ is defined in \eqref{eq:Cxdef}.
 For two different points $z,w\in V$, write $d=|z-w|$. 
\begin{prop}\label{prop:uptocb2}
For $0<r,r'<d/100$ such that $B(z,r),B(w,r')\subset V$, there exists some $M_1(V,d,r,r') > 0$ such that for all $m\geq M_1$, we have
\begin{equation}\label{eq:2b}
P\Big(\eta_m \cap B(z,r) \neq \emptyset,\;\eta_m \cap B(w,r') \neq \emptyset\Big)\asymp C_V  \left(\frac{rr'}{\vert z - w\vert}\right)^{3-\beta}
\end{equation} 
and
\begin{equation}\label{eq:1b1p}
P\Big(\eta_m \cap B(z,r) \neq \emptyset,\;w_m\in \eta_m\Big)\asymp C_V  \left(\frac{r}{m \vert z - w\vert}\right)^{3-\beta}.
\end{equation}
Equivalently, with the same assumption as above,
\begin{equation}\label{eq:cb1}
P\Big(\KK \cap B(z,r) \neq \emptyset,\;\KK \cap B(w,r') \neq \emptyset\Big)\asymp C_V  \left(\frac{rr'}{\vert z - w\vert}\right)^{3-\beta}.
\end{equation}
\end{prop}
We omit the proof as it is essentially the same as those of Propositions 5.2 and 5.3 of~\cite{Natural}. Note that the independence of $M_1$ on $z$ and $w$ follows from the equivalence of the convergence under L\'evy metric and the weak convergence under $d_{\text{Haus}}$-metric.

\subsection{Asymptotic independence for 3D LERW}

In this subsection, we will briefly review the asymptotic independence of three-dimensional LERW in the form of a decoupling result, which serves as the core argument in \cite{Natural}. In essence, this result states that with a small error, the ratio of the hitting probability of two sets that are not too close to each other and the hitting probability of two points roughly the same location can be decoupled into the product of the ratio of  hitting probabilities of the respective sets at a reference location and the hitting probability of a point at this location.

We now state it more precisely.  Recall the definition of $\cD^o_k$ below \eqref{eq:dnodef}. Let $V\in\cD^o_k$ for some $k\in \Z^+$. Let $\hat{x}=(1/2,0,0)$ be a reference point. For $v\in\Rt$ and a set $S\subset \Rt$, write $S(v)\coloneqq v+S=\{v+x; x\in S\}$ for $S$ shifted by $v$.

\begin{prop}\label{prop:decoup}
For $k\in\Z^+$, consider $m \in \mathbb{R}^+$ such that $m\geq 2^{k^4}$ and
\begin{equation} \label{eq:anSdef} 
 a_{m}(S)=\frac{P\big(S(\hat{x})\cap \eta_{m}\neq \emptyset\big)}{P\big(\hat{x}\in\eta_{m}\big)}, 
\end{equation}
for any $S \subset m^{-1} \mathbb{Z}^3 $.
Let $ z, w \in V \cap m^{-1}\Z^3$ such that $|z-w|\geq 2^{-k^2}$. 
Let 
\[
  S_1,S_2\subset \left[-2^{-k^4},2^{-k^4}\right]^3\bigcap m^{-1} \Z^3.
\]
Then,
\begin{equation}\label{eq:ppbb}
\frac{P\Big(S_1(z)\cap \eta_{m}\neq \emptyset,\; S_2(w) \cap\eta_{m}\neq \emptyset\Big)}{P\big(z,w\in \eta_{m}\big)} =\Big[  1+O\big(2^{-k^2}\big) \Big] a_{m}(S_1)a_{m}(S_2).
\end{equation}
Similarly,
\begin{equation}\label{eq:bpbb}
\frac{P\Big(S_1(z)\cap \eta_{m}\neq \emptyset,\; w\in \eta_{m}\Big)}{P\big(z,w\in \eta_{m}\big)} =\Big[  1+O\big(2^{-k^2}\big)\Big] a_{m}(S_1).
\end{equation}
\end{prop}

 This proposition follows from a minor modification of the proof for (6.6) and (6.7) of \cite{Natural} where $S_1$ and $S_2$ are restricted to certain boxes. 
See also (6.8) and (6.9), ibid. Note that this decoupling result holds for non-dyadic scales too. 

\section{A continuity result}\label{se:continuity}
In this section, we will present the following property of $\KK$, the scaling limit of 3D LERW:  a.s., $\KK$ cannot hit the boundary of a ball but do not enter it. This property will be quite helpful in Section \ref{se:proof} as it allows us to switch between the discrete and continuum when it comes to hitting probability estimates. 
\begin{prop}\label{prop:continuity}
For any $x\in\mathbb{D}$ and $r>0$ such that $B(x,r)\subset\mathbb{D}$, 
$$
P\Big(\KK \cap B(x,r)\neq\emptyset\mbox{ but }\KK \cap B^\circ(x,r)=\emptyset\Big)=0.
$$
\end{prop}
As this result similar to Proposition 7.2 of \cite{Natural}, where the authors consider boxes instead of balls, we will only sketch the proof of Proposition \ref{prop:continuity} here.

\begin{figure}[h] 
\centering
\includegraphics{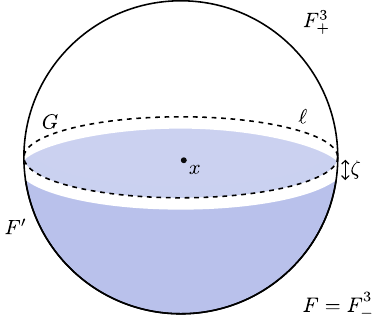}
\caption{ The equator $\ell$ divides $\partial B$ into two shells: the northern hemisphere $F_{+}^3$ and the southern one $ F = F_{-}^3 $.
The shaded area represents the set $F' \subset \partial F$, consisting of points $x \in F$ in the southern hemisphere that are at a distance greater than $ \zeta =  \varepsilon^{b/20}$ from the equatorial line $\ell$. 
}  \label{fig:sphere}
\end{figure}

\begin{proof}[Proof sketch of Proposition~\ref{prop:continuity}]
Write $x = (x_1, x_2, x_3) \in \mathbb{D}$ and $B = B (x, r)$ for short. We may assume that $m \coloneqq \text{dist} \big( B, \{ 0 \} \cup \partial \mathbb{D} \big) > 0$. Finally, we fix $0<\varepsilon <10^{-5} r$.

Let $\eta = \eta_{2^n}$ be the LERW on $\mathbb{D}_{2^n}$ from the origin. For $a > b \ge 10^{6}$, we set 
\begin{equation*}
\widehat{B} = \widehat{B}_{a} = B \big( x, r - \varepsilon^{a} \big) \ \  \ \ \ \text{ and } \ \ \ \ \  \widetilde{B} = \widetilde{B} _{b} = B \Big( x + \big( 0,0, \varepsilon^{b}  \big), r \Big).
\end{equation*}

We next consider six subsets $F^{i}_{+}, F^{i}_{-}$ ($i=1,2,3$) of $\partial B$ defined by
\begin{equation*}
F^{i}_{+} = \big\{ (y_1, y_2, y_3 ) \in  \partial B \ \big| \ y_i \ge x_i \big\} \ \  \ \ \ \text{ and } \ \ \ \ \ F^{i}_{-} = \big\{ (y_1, y_2, y_3 ) \in  \partial B \ \big| \ y_i \le x_i \big\}.
\end{equation*}
Note that $\partial B = \bigcup_{i =1}^{3}  F^{i}_{+}  \cup F^{i}_{-}$. Let $\tau $ be the first time that $\eta$ hits $B$. Then we have 
\begin{equation}\label{darui}
\begin{split}
P \big( \eta \cap B \neq \emptyset \text{ and } \eta \cap \widehat{B} = \emptyset \big) 
&\;\le \sum_{i=1}^{3} P \big( \eta \cap B \neq \emptyset, \ \eta \cap \widehat{B} = \emptyset, \ \eta (\tau) \in F^{i}_{+} \big)\\ +  \sum_{i=1}^{3} &\;P \big( \eta \cap B \neq \emptyset, \ \eta \cap \widehat{B} = \emptyset, \ \eta (\tau) \in F^{i}_{-} \big).
\end{split}
\end{equation}
We now bound the RHS of \eqref{darui}. By symmetry, we can deal with the case of 
$F = F^{3}_{-}$ only. See Figure~\ref{fig:sphere} for reference. Let
\[
\ell = \big\{ y = (y_1, y_2, y_3) \in \partial B \ \big| y_3 = x_3 \big\}
\]
be the ``equator'' of $B$. It is not difficult to show that one can choose $a > b \ge 10^{6}$ such that the following holds:
\begin{itemize}
\item[(i)] If we define 
\begin{equation*}
F' = \big\{ y \in F \ \big| \ \text{dist} (y, \ell ) \ge \varepsilon^{b /20} \big\},
\end{equation*}
then for all $y  \in F'$ it holds that 
\begin{equation}\label{darui-1}
y + (0,0, \varepsilon^{b} ) \in \widehat{B}_{a}.
\end{equation}
We write 
\begin{equation*}
F'' = \big\{ y  + (0,0, \varepsilon^{b} ) \ \big| \ y \in F' \big\} \subset \widetilde{B}_{b}
\end{equation*}
which is contained in the ``southern hemisphere'' of $\widetilde{B}_{b}$.

\item[(ii)] For any continuous simple curve $\lambda = \lambda [0, u]$ satisfying that $\lambda (0) = 0$, $\lambda (u) \in F''$  and $u$ is the first time that $\lambda$ hits $\widetilde{B}_{b}$, it follows that $\lambda$ must hit $\widehat{B}_{a}$ (in particular, it hits $B$). If we write $t$ for the first time that $\lambda$ hits $B$, then it holds that 
\begin{equation}\label{darui-2}
\lambda (t) \in F \cup \big\{ y \in \partial B \ \big| \ \text{dist} (y, \ell ) \le \varepsilon^{b/5} \big\}.
\end{equation}

\end{itemize}

We fix $a$ and $b$ which satisfy the conditions (i) and (ii) above. Let $\zeta = \varepsilon^{b/20}$ and define
\begin{equation*}
G = \big\{ y \in \partial B \ \big| \ \text{dist} (y, \ell ) \le \zeta \big\}.
\end{equation*}
We now calculate the probability that $\eta$ hits $G$. 
Let $\{ y_j \}_{j=1}^{J} \subset \ell$ be a set of points such that \begin{equation*}
G \subset \bigcup_{j=1}^{J}  B (y_{j} , 2 \zeta ) \ \ \ \ \ \text{ and } \ \ \ \ \ J \asymp r/\zeta.
\end{equation*}  
By \eqref{eq:oneballd},
$P \big( \eta \cap B (y_{j} , 2 \zeta ) \neq \emptyset \big) \le C \zeta^{3 - \beta } 
$
for each $j$. 
Since $2- \beta > 0$, we see that 
\begin{equation}\label{darui-3}
P \big( \eta \cap G \neq \emptyset \big) \le \sum_{j=1}^{J} P \big( \eta \cap B (y_{j} , 2 \zeta ) \neq \emptyset \big) \le C r \zeta^{2 - \beta } \to 0,  \  \text{ as } \varepsilon \to 0.
\end{equation}

Once we have \eqref{darui-3}, we can follow the same scenario as in the proof of \cite[Proposition 7.2]{Natural}. The steps are as follows.

\begin{itemize}
\item[Step 1.] We let $\widetilde{\eta}= \eta + (0,0, \varepsilon^{b})$ and write $\widetilde{\tau}$ for the first time that $\widetilde{\eta}$ hits $B$. The conditions (i) and (ii) ensure that 
\begin{equation*}
\eta \cap B \neq \emptyset, \ \eta (\tau) \in F \text{ and } \eta \cap G = \emptyset   \Rightarrow \widetilde{\eta} \cap \widehat{B}_{a} \neq \emptyset \text{ and } \widetilde{\eta} (\widetilde{\tau} ) \in F.
\end{equation*} 
This implies
\begin{equation}\label{darui-4}
P \big( \widetilde{\eta} \cap \widehat{B}_{a} \neq \emptyset, \ \widetilde{\eta} (\widetilde{\tau} ) \in F \big) \ge P \big( \eta \cap B \neq \emptyset, \ \eta (\tau) \in F  \big) - C r \zeta^{2 - \beta}.
\end{equation}

\item[Step 2.] The coupling argument via a cut time to prove  \cite[Equation (7.10)]{Natural} still works here. The same argument shows that there exists $\delta > 0$ such that 
\begin{equation}\label{darui-5}
\Big| P \big( \widetilde{\eta} \cap \widehat{B}_{a} \neq \emptyset, \ \widetilde{\eta} (\widetilde{\tau} ) \in F \big)  - P \big( \eta \cap \widehat{B}_{a} \neq \emptyset, \ \eta ( \tau ) \in F \big) \Big| \le C \varepsilon^{\delta b }.
\end{equation}

\item[Step 3.] By \eqref{darui-4}, we have 
\begin{align*}
P \big( & \eta \cap \widehat{B}_{a} \neq \emptyset, \ \eta (\tau ) \in F \big)\; \\ = \; & P \big( \eta \cap B \neq \emptyset, \ \eta (\tau ) \in F \big) - P \big( \eta \cap B \neq \emptyset, \ \eta (\tau ) \in F, \ \eta \cap \widehat{B}_{a} = \emptyset \big) \\
\le \;& P \big( \widetilde{\eta} \cap \widehat{B}_{a} \neq \emptyset, \ \widetilde{\eta} (\widetilde{\tau} ) \in F \big) + C r \zeta^{2 - \beta}  - P \big( \eta \cap B \neq \emptyset, \ \eta (\tau ) \in F, \ \eta \cap \widehat{B}_{a} = \emptyset \big).
\end{align*}
\end{itemize}
It follows from \eqref{darui-5} that 
\begin{align*}
P \big(& \eta \cap B \neq \emptyset, \ \eta (\tau ) \in F, \ \eta \cap \widehat{B}_{a} = \emptyset \big) \; \\
\le\; &P \big( \widetilde{\eta} \cap \widehat{B}_{a} \neq \emptyset, \ \widetilde{\eta} (\widetilde{\tau} ) \in F \big) - P \big( \eta \cap \widehat{B}_{a} \neq \emptyset, \ \eta (\tau ) \in F \big) + C r \zeta^{2 - \beta} \\
\le\;& C \varepsilon^{\delta b } + C r \zeta^{2 - \beta}.
\end{align*}
Combining this with \eqref{darui}, we see that 
\begin{equation}\label{darui-6}
P \big( \eta \cap B \neq \emptyset, \ \eta \cap \widehat{B}_{a} = \emptyset \big) \le C \varepsilon^{c b }
\end{equation}
for some $c, C \in (0, \infty )$. Since we can couple $\eta = \eta_{2^n}$ and ${\cal K}$ in the same probability space such that the Hausdorff distance between them converges to $0$ as $n \to \infty$, it follows from \eqref{darui-6} that 
\begin{equation}\label{darui-7}
P \big( {\cal K} \cap B \neq \emptyset, \ {\cal K} \cap \widehat{B}_{a} = \emptyset \big) \le C \varepsilon^{c b }.
\end{equation}
This finishes the proof.
\end{proof}

\section{Sharp one-point function estimates}\label{se:onept}
In this section, we will give proofs for Theorems \ref{thm:oneptnew} and \ref{thm:oneballnew}, the sharp asymptotics of one-point function for 3D LERW. This section is organized as follows: in Subsection \ref{sec:4.1} we prove Theorem \ref{thm:oneptnew} and then mention briefly how to adapt the proof to derive Theorem \ref{thm:esnew}, the asymptotics of non-intersection probabilities of LERW and SRW in arbitrary scales. In Subsection \ref{sec:4.2} we derive Theorem \ref{thm:oneballnew}.

\subsection{Proof of Theorem \ref{thm:oneptnew}}\label{sec:4.1}
We first state a result on functional equations.
\begin{prop}\label{prop:funceq}
Let $f:\mathbb{R}\to\mathbb{R}^+$ be a function such that
\begin{itemize}
 \item (Boundary value) there exist some $A$ and $B$ such that
\begin{equation}\label{eq:trivial}
f(0)=A,\quad f(1)=AB.
\end{equation}
\item (Local logarithmic additivity)
  For all $r, s \ge 0$ satisfying $\lfloor r \rfloor = \lfloor s \rfloor$, 
\begin{equation}\label{eq:funceqlin}
f(r)f(s)=f(r+s)f(0).
\end{equation}
 \item  (Comparability) There exists $C\in(0,\infty)$ such that for all $|r-s|\leq 1$,
\begin{equation}\label{eq:comparability}
\left|{f(r)}/{f(s)}\right| \in \left(\frac{1}{C},C\right).
\end{equation}
\end{itemize}
Then, for all $r>0$,
\begin{equation}\label{eq:fcteqres}
    f(r)=AB^{r}.
\end{equation}
\end{prop}
\begin{proof}
    By an induction argument, it is straightforward to check that~\eqref{eq:trivial} and~\eqref{eq:funceqlin} imply that~\eqref{eq:fcteqres} holds for all integers of the form $r = 2^{n} >0$. If for some $r_0\in (0,1]$ the identity~\eqref{eq:fcteqres} does not hold,
    then by iterating~\eqref{eq:funceqlin} one can find a large integer $2^M$ such that $f(2^Mr_0)/f(\lfloor 2^Mr_0\rfloor)\notin (1/C,C)$, contradicting \eqref{eq:comparability}!
    It follows that~\eqref{eq:fcteqres} holds for all $r\in (0,1]$.
   
    From here, we extend~\eqref{eq:fcteqres} to all positive $r = \cup_{m \in \mathbb{N}} (0, m]$ by induction on $m$. Assuming that~\eqref{eq:fcteqres} holds for all $r \in (0,m]$, let $t \in (m,m+1]$. Applying \eqref{eq:funceqlin} to the decomposition $t = t/2 + t/2$, we get that $t$ also satisfies~\eqref{eq:fcteqres}.
\end{proof}

We now turn back to one-point functions.

Our first tool is a decomposition of the one-point probability into a Green's function and a certain non-intersection probability. Let $X$ be a simple random walk started from $x_{m}$ conditioned to hit $0$ before exiting $\mathbb{D}_{m}$, and let $Y$  be a simple random walk started from $y_{m}$ stopped at exiting $\mathbb{D}_{m}$ (we denote this stopping time by $T$). We write $P_{x, y}^{m}$ for the joint law of $(X,Y)$.
(Note that we omit the scaling for the starting points $x_{m}$ and $y_{m}$, to lighten the notation).
From~\cite[Lemma 5.1]{Escape} we have that
\begin{equation} \label{eq:decompGreen} 
 P(x_{m}\in \eta_{m}) = G_{\DD_{m}}(0,x_{m}) P^{m}_{x,x}(\LE(X) \cap Y[1,T] =\emptyset) .
\end{equation}
This last identity follows from~\cite[Lemma 5.1]{Escape}. A version of this decomposition was proved for homogeneous Markov chains in~\cite[Proposition 5.2]{BarlowMasson}.

For the rest of this subsection, fix $x\in \Dmo$ and for any $r\in\Rp$, write
\begin{equation} \label{eq:onePointExp}
  z(r)=P(x_{2^r}\in \eta_{2^r}).
\end{equation}

\begin{prop}\label{prop:funceqss} For any $n \in\mathbb{Z}^+$ and all $r,s\in[0,1]$,
    \begin{equation}\label{eq:funceqess}
    \frac{z(r+s+n)z(n)}{z(r+n)z(s+n)}=1+O\big(d_x^{-c}2^{-\delta n}\big).
    \end{equation}
\end{prop}

\begin{proof}[Proof of Proposition \ref{prop:funceqss}]
Let $G_{\DD}(\cdot,\cdot)$ stand for the Green's function of the simple random walk in $\DD$ killed upon exiting $\DD$. Given $r>0$ and $x,y\in \mathbb{D}_{2^r}$, recall that $P^{2^r}_{x,y}$ denotes the joint law of $(X,Y)$ defined just above~\eqref{eq:decompGreen}.
As in~\eqref{eq:decompGreen}, the probability in~\eqref{eq:onePointExp} that $x_{2^r}$ is on $\eta_{2^r}$ satisfies 
\begin{equation} \label{eq:decompZ} 
  z({r})=G_{\DD_{2^r}}(0,x_{2^r}) P^{2^r}_{x,x}(\LE(X) \cap Y[1,T] =\emptyset) .
\end{equation}

By the same technique as in~\cite{Escape}, we need to change the starting point of $X$'s and $Y$'s as well as truncate the tail of $X$. We pick a small $q\in(0,1/10)$ as in  \cite[Proposition 3.4]{Escape}. For any $n\in\mathbb{Z}^+$ and $r,s\in[0,1]$, let 
\[ 
  x_1=(2^{-5qn-r-s},0,0),\;y_1=-x_1,\; x_2=2^s x_1,\;y_2=2^s y_1.
\]
For $x_i,y_i\in\mathbb{R}^3$, $i=1,2$, 
write 
\[ 
  W_{x_i,y_i}(r)\coloneqq P^{2^r}_{(x-x_i),(x+y_i)}\big(\LE(X) \cap Y =\emptyset\big)
\]
for the non-intersection probability of the loop-erasure of a SRW started from (the discretization of)  $x-x_i$ conditioned to hit $0$ before exiting $\mathbb{D}_{2^r}$ and SRW $Y$ starting from (the discretization of) $x+y_i$ and stopped at exiting $\mathbb{D}_{2^r}$.

In the following we use $a(m) \simeq b(m)$ to represent $a(m)=b(m)\big(1+O(m^{-c})\big)$ for constants universal for $r,s\in[0,1]$. Using the same technique as in \cite{Escape}, we have (note that constants in $\simeq$ are universal for $r,s\in[0,1]$)
\begin{equation}\label{eq:chg1}
\frac{z(r+s+n)}{z(r+n)}\bigg/\frac{z(s+n)}{z(n)} \simeq\frac{W_{x_1,y_1}(n+r+s)}{W_{x_2,y_2}(n+r)}\bigg/\frac{W_{x_1,y_1}(n+s)}{W_{x_2,y_2}(n)}.
\end{equation}
Repeating the proof of \cite{Escape} but also involving the scaling invariance of the scaling limit of LERW, see Subsection 6.1 of \cite{Kozma}), we have
\begin{equation}\label{eq:chg2}
W_{x_1,y_1}(n+r+s)\simeq W_{x_1,y_1}(n+s)\quad\mbox{ and }\quad W_{x_2,y_2}(n+r)\simeq W_{x_2,y_2}(n).
\end{equation}
The claim \eqref{eq:funceqess} now follows from \eqref{eq:chg1} and \eqref{eq:chg2}.
\end{proof}

\begin{proof}[Proof of Theorem \ref{thm:oneptnew}]
We start with \eqref{eq:oneptnew}. We first note that in \cite[Theorem 1.1]{Escape} the claim \eqref{eq:oneptnew} has already been verified for $m=2^n$, $n\in\Zp$. 
    Moreover, the proof of \cite[Theorem 1.1]{Escape} is valid for $m\in[2^n,2^{n+1})$, with $n\in\Zp$. 
    We write $m=2^{n+r}$ for $r\in[0,1)$, and following the steps in \cite{Escape},
    by an application\footnote{Note that one can pick constants in that proof uniformly for $s\in[0,1]$, for the same reason as the proof of Proposition \ref{prop:Hanyscale}.} of 
\cite{Kozma} for the scaling limit of LERW in $2^s\mathbb{D}$ for $s\in[0,1]$, we obtain that there exist $c,\delta>0$ and some $f:[0,1)\to \mathbb{R}^+$, such that for any $r\in[0,1)$,
\begin{equation}\label{eq:zn}
z(n+r)= f(r) 2^{-(3-\beta)n}\Big[1+O\big(d_x^{-c}2^{-\delta n}\big)\Big]
\end{equation}
where $d_x=\min(|x|,1-|x|)$ was defined in Theorem \ref{thm:oneptnew} and the constants in $O(\cdot)$ do not depend on $r$.

 To show \eqref{eq:oneptnew}, it now remains to show that there exists some $c$ such that for all $r\in[0,1)$,
\begin{equation}\label{eq:funceq}
f(r)=c 2^{-(3-\beta) r}.
\end{equation}
To check \eqref{eq:funceq}, we apply Proposition~\ref{prop:funceq}. With this purpose, we now extend the definition of $f$ to the whole positive half-axis: for $s= t+r$ such that $t \in\Z$ and $r\in[0,1)$, let
\begin{equation}
    f(s)=f(r)2^{-(3-\beta)t}.
\end{equation}
Now we verify that the function $f$ satisfies the conditions of Proposition~\ref{prop:funceq}. The claim \eqref{eq:trivial} is trivial. Note that condition \eqref{eq:funceqlin} follows from Proposition~\ref{prop:funceqss}, as the constants in the $O(\cdot)$ expression in \eqref{eq:zn} does not depend on $r$. To check condition \eqref{eq:comparability}, we apply~\cite[Proposition 8.2]{Growth}, \cite[Proposition 8.5]{Growth} and \cite[Corollary 1.3]{Escape}. This finishes the proof of \eqref{eq:oneptnew}.

Recall the definition of $g(\cdot)$ from \eqref{eq:oneptnew}. The continuity of $g(\cdot)$ follows from~\cite[Lemma 5.4]{Natural}. The rotation-invariance of $g(\cdot)$ is a corollary of Theorem \ref{thm:oneballnew} and the rotation-invariance of $\KK$.
\end{proof}

Finally, we note that similar techniques as in the proof above also allow us to obtain the  sharp asymptotics of non-intersection probabilities of LERW and SRW in any scale, which is an improvement of \cite[Theorem 1.2]{Escape}.

Let $S = S[0,\infty)$ be a discrete-time simple random walk on $\mathbb{Z}^3$ started from $0$ and $T_m \coloneqq \inf\{T\geq 0;|S(T)|\geq m\}$ be the first exit time of a ball of radius $m$. 
Let $\gamma_m \coloneqq \LE(S[0,T_m])$ be the LERW of SRW truncated upon leaving a ball of radius $m$. Recall that $\beta\in (1,5/3]$ stands for the growth exponent of 3D LERW.
Let $S'$ be an independent copy of $S$ and define $T'_m$ accordingly. Write $P$ for the joint law of $S$ and $S'$.
Let
\begin{equation}
    \Es(m) \coloneqq P\Big(\gamma_m\cap S'[1,T'_m]=\emptyset\Big)
\end{equation} 
stand for the non-intersection probability of $\gamma_m$ and $S'[1,T'_m]$ and recall the definition of the universal constant $c > 0$ from \cite[Theorem 1.2]{Escape}. 
\begin{theorem}\label{thm:esnew}
There exist a universal constants $c > 0$ and $\delta>0$ such that for all $m>0$,
\begin{equation}\label{eq:esnew}
\Es(m)=c m^{-2+\beta} \Big[ 1+O\big(m^{-\delta}\big)\Big].
\end{equation}
\end{theorem}
We omit the proof of this theorem as the same idea as above also works perfectly for asymptotics non-intersection probability.

  The sharp estimate on the non-intersection probability allows us to improve the upper bound on the one-ball hitting probability given in~\cite[Lemma 7.1]{SS}. We state this improved bound in the following lemma, as it will be useful to show our sharp bound for the one-ball hitting probability.

  \begin{lemma} \label{lemma:upperBoundOBall} 
  There exists a universal constant $C < \infty$ such that for all $ m \in \mathbb{R}^+ $, $s \in (0,1)$ and each $x \in \mathbb{D}_m $ with $d_x > 2s$, we have 
  \[
    P \left( \eta_m \cap B (x,s) \neq \emptyset \right) \leq C \left( \frac{s}{\vert x \vert} \right)^{3 - \beta}.
  \]
  \end{lemma}

  The proof of Lemma~\ref{lemma:upperBoundOBall} is the same as that of~\cite[Lemma 7.1]{SS}, but with the bounds for the escape probability from~\cite[Lemma 7.2.2]{Growth} replaced by those in Theorem~\ref{thm:esnew}.

\subsection{Proof of Theorem \ref{thm:oneballnew}}\label{sec:4.2}

In this subsection, we deal with the asymptotics of the probability that a LERW passes through a ball and prove Theorem \ref{thm:oneballnew}, which concerns the one-point function for the scaling limit. The general strategy of the proof is somewhat similar to that used in Theorem~\ref{thm:oneptnew} and in~\cite[Theorem 1.1]{Escape}. However, some special care is needed in the decomposition of the path, and a new decoupling technique plays a crucial role.

We start with a fixed point $\hat{x} =(1/2,0,0)$ which serves as a reference point.

\begin{prop}\label{one-ball-re}
For each $s\in\Rp$, there exists some $m_0=m_0(s)\in \Rp$ such that for all $m>m_0$, there exists some $c_2, \delta>0$ satisfying
\begin{equation}\label{eq:one-ball-re}
    P\Big( \eta_m \cap B(\hat{x},s) \neq \emptyset\Big)=c_2 s^{3-\beta} \Big[ 1+O\big(s^{\delta}\big)\Big].
\end{equation}
\end{prop}

Before we proceed to the proof of Proposition~\ref{one-ball-re}, we present two of its consequences, the second one being a proof of Theorem~\ref{thm:oneballnew}.

The first consequence is on the asymptotic behavior of the function $a_m (\cdot)$ for $m \in \mathbb{R}^+$. Recall that $a_m (\cdot)$ is defined in \eqref{eq:anSdef}.  For $s > 0$, let 
\[
    A_m (s) = a_m \left( B(s) \cap m^{-1} \mathbb{Z}^3 \right) = \frac{ P \left(  \eta_{m} \cap B(\hat{x}, s)  \neq \emptyset \right) }{P (\hat{x} \in \eta_{m})}.
\]

From Theorem~\ref{thm:oneptnew} and Proposition~\ref{one-ball-re}, we easily obtain the following corollary. This result will be useful later in our study of two-point estimates (Proposition~\ref{prop:crux2pt} and Theorem~\ref{thm:two-point}).

\begin{cor}  \label{cor:asymp}
  Let $\delta > 0$ be the constant in Proposition~\ref{one-ball-re}.
  There exists a function $f : \mathbb{R}^{+} \to \mathbb{R}^+$ satisfying
  \[
    f(s)=c s^{3-\beta}\Big[  1+O\big(s^{-\delta }\big)\Big] \qquad  \text{ for } s > 0,
  \]
    and such that, for every $s > 0$,
    \begin{equation}\label{eq:asymp}
    \lim_{m \to \infty} \frac{A_m (s)}{m^{3-\beta}} = f(s) .
    \end{equation}
\end{cor}

We now comment on the proof of Theorem~\ref{thm:oneballnew}. The strategy consists of checking that the one-point function appearing in Theorem~\ref{thm:oneballnew} is a constant multiple of the discrete one-point function established in Theorem~\ref{thm:oneptnew}. We first compare hitting probabilities for balls centered at a point  $x \in \mathbb{D}$ and at the reference point $\hat{x}$.

\begin{prop}\label{prop:rotation}
For any $x\in \mathbb{D}$ and $0<s<d_x=\min(|x|,1-|x|)$, there exists some $m_0=m_0(s)\in \Rp$ such that for all $m>m_0$,
\begin{equation}
\frac{P\Big(\eta_m \cap B(\hat{x},s) \neq \emptyset\Big)}{P\Big((\hat{x})_m\in \eta_m\Big)} =\frac{P\Big(\eta_m \cap B(x,s) \neq \emptyset\Big)}{P\Big(x_m\in \eta_m\Big)}\Big[ 1+O\big(d_x^{-c}m^{-\delta}\big)\Big]
\end{equation}
for some $c,\delta>0$.
\end{prop}
We omit the proof of Proposition~\ref{prop:rotation} as it follows from the same coupling technique as the proof of Proposition \ref{one-ball-re}, with the only extra ingredient being the rotation invariance of the scaling limit of 3D LERW.

Finally, Theorem~\ref{thm:oneballnew} follows from passing the results of Propositions~\ref{one-ball-re} and~\ref{prop:rotation} in the continuum thanks to Proposition~\ref{prop:continuity}.

We now work towards the proof of Proposition~\ref{one-ball-re}. First, recall that $S_m$ is a simple random walk on $m^{-1}\mathbb{Z}^3 $ starting at $0$ and stopped on its first exit of $\mathbb{D}_m$, and  $ \eta_m  = \LE (S_m ) $.
Let us introduce additional notation that will be used for the rest of this section.
\begin{itemize}
  \item Let 
  \begin{equation} \label{eq:boundary}
  \partial_{i} B( \hat{x}, s ) =
  \left\{ y \in m^{-1} \mathbb{Z}^{3} \cap B ( \hat{x} , s)  \ \Big| \ 
  \begin{array}{c}
  \text{There exists } \\ z \in m^{-1} \mathbb{Z}^{3} \cap B ( \hat{x} , s)^{c}  \\ \text{ such that } |y -z | = m^{-1}  
  \end{array}
  \right\}
  \end{equation}
  be the (discrete) inner boundary of $B ( \hat{x} , s)$.

  \item $S^{1}$ and $ S^{2}$ stand for independent simple random walks in $m^{-1} \mathbb{Z}^{3}$ whose probability law is $P^{v, w}_{1,2}$ when we assume $S^{1} (0) = v$ and $S^{2} (0) = w$.  

  \item Let 
  \begin{equation} \label{eq:firsthit0}
    t^{1}_{0} =  \inf \{ j \geq 0  \ | \ S^{1} (j) = 0 \}
  \end{equation}
  be the first time that $S^{1}$ hits the origin and write $$\gamma^0={\rm LE}(S^1[0,t^{1}_{0}])$$ for short.

  \item $T^{i}$ is the first time that $S^{i}$ leaves $\mathbb{D}$ for $i=1, 2$, i.e.
  \begin{equation} \label{eq:firstexit}
    T^{i} = \inf \{ j \geq 0 \ | \ S^{i}(j) \notin \mathbb{D} \}.
  \end{equation}

  \item Let 
  \begin{equation} \label{eq:xi}
  \xi = \max \Big\{ j \geq 0 \ \Big| \  S (j) \in B (\hat{x} , s)  \Big\}
  \end{equation}
  be the last time that the simple random walk  $S_m$ stays in $B (\hat{x} , s)$. 

  \item Let 
  \begin{equation} \label{eq:u1}
  u_{1} = \max \Big\{ j \in [0, u_2 ] \ \Big| \   \gamma^0 (j) \in B ( \hat{x} , s)  \Big\}
  \end{equation}
  be the last time that the loop-erasure of $S^{1}$ stays in $B ( \hat{x} , s)$.

  \item Write
  \begin{equation} \label{eq:u2}
    u_{2} = \text{len} \Big( \gamma^0 \Big) 
  \end{equation}
  for the length of the loop-erasure.

  \end{itemize}

\begin{proof}[Proof of Proposition \ref{one-ball-re}]

Let $  y \in \partial_i B (\hat{x}, s) $. 
If the random walk $S_m$ hits the vertex $y$, note that (up to time reversal) $S_m$  is equivalent to the concatenation of two independent simple random walks starting at $y$, one going forward in time, $S^2$, and the other going backward $S^1 [0, t^{1}_{0}] $,
conditioned on the event that $S^1$ hits $0$ before exiting $\mathbb{D}$, i.e. $ t^{1}_{0} < T^{1} $. 
If we further assume that the last visit to $B (\hat{x},s)$ occurs at time $\xi$ with $S_m (\xi ) = y$, we obtain:
\[
  P \left (   S_m (\xi ) = y \right) 
  =  
  P^{y,y}_{1,2} \left(  t^{1}_{0} < T^{1},  S^{2} [1, T^{2} ]  \cap B ( \hat{x} , s) = \emptyset    \right).
\]
By definition of the loop-erasure, writing
$$
S^2_+=S^{2} [1, T^{2} ] 
$$
for short, we have
\begin{equation} \label{mendoi-00}
P \left (   S_m (\xi) = y , \eta_m \cap B  (\hat{x}, s ) \neq \emptyset\right) 
  =  
  P^{y,y}_{1,2}
  \Bigg(  
    \begin{array}{c}
    t^{1}_{0} < T^{1}, \ \gamma^0 [u_{1}, u_{2} ]  \cap S^2_+ = \emptyset, \\
    S^2_+ \cap B ( \hat{x} , s) = \emptyset   
    \end{array}
  \Bigg).
\end{equation}
Applying a last-exit decomposition and using~\eqref{mendoi-00}, we obtain the following expression for the left-hand side of~\eqref{eq:one-ball-re}: 
\begin{align}\label{mendoi}
   P \left( \eta_m \cap B  (\hat{x}, s ) \neq \emptyset \right) 
   &= \sum_{ y \in \partial_i B (\hat{x}, s) } P \left(  S_m (\xi) = y , \eta_m \cap B  (\hat{x}, s ) \neq \emptyset \right) \nonumber \\
   &= \sum_{y \in \partial_{i} B(\hat{x}, s ) }  P^{y,y}_{1,2}
  \Bigg(  
    \begin{array}{c}
    t^{1}_{0} < T^{1}, \ \gamma^0 [u_{1}, u_{2} ]  \cap S^2_+ = \emptyset, \\
    S^2_+ \cap B ( \hat{x} , s) = \emptyset   
    \end{array}
  \Bigg).
\end{align}
Here, we use the notation introduced above, between~\eqref{eq:boundary} and~\eqref{eq:u2}.

We would like to say that the probability in the sum of~\eqref{mendoi} actually does not depend on $y \in \partial_{i} B(\hat{x}, s ) $. However, unfortunately, this is not the case. To deal with this issue, we use a coupling method as follows.

First, we will change the length of the loop-erasure of $S^1$ in the events appearing in~\eqref{mendoi}. 
We will replace $u_{1}$ with the time $u_{1}'$ as defined below in~\eqref{eq:newlength}.
Let $\alpha = 2 - \beta  \in [1/3, 1)$. We set
\begin{equation} \label{eq:delta0}
  \delta_{0} = (1- \alpha )/2 >0
\end{equation}
 such that $1- \delta_{0} $ is the middle point of $\alpha$ and $1$. We let
\begin{align} \label{eq:newlength}
&v_{1} = \inf \Big\{ j \ge 0 \ \Big| \ \gamma^0 (j) \notin B \big(\hat{x} , s^{\delta_{0}} \big) \Big\} \text{ and } \nonumber \\
&u_{1}' =  \max \Big\{ j \in [0, v_1 ] \ \Big| \   \gamma^0 (j) \in B (\hat{x} , s)  \Big\}.
\end{align}
The difference between $u_{1}$ and $u_{1}'$ is that $u_{1}$ considers the time interval  $[ 0, u_2] = [ 0, \text{len} (\gamma^0)]$ (with $u_{2}$ is as defined in~\eqref{eq:u2}), while $u_{1}'$ considers a possibly shorter time interval $[0 , v_{1}]$. See Figure~\ref{fig:LERWcomp} for a graphical reference.

\begin{figure}[h] 
\centering
\includegraphics{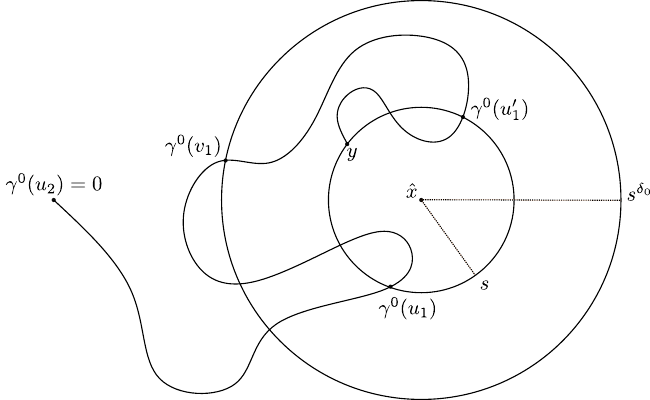}
\caption{An illustration of the event $ \{ u_1 \neq u_1' \} $.
}  \label{fig:LERWcomp}
\end{figure}

Note that the event $ \{ u_{1} \neq u_{1}' \} $ is contained in the event that the LERW path $\gamma^0  [v_1, u_2]$intersects the ball $B(\hat{x},s)$. Then, by Lemma~\ref{lemma:upperBoundOBall} and~\cite[Proposition 1.5.10]{Lawb}  
\begin{equation}\label{eq:bound1}
P^{y,y}_{1,2} \Big( t^{1}_{0} < T^{1}, \ u_{1} \neq u_{1}', \  S^2_+ \cap B (\hat{x} , s) = \emptyset   \Big) \le C s^{- \delta_{0}} m^{-2},
\end{equation}
and by adding~\cite[Proposition 3.2.2]{Hausdorff} for the lower bound
\begin{equation}\label{eq:bound2}
P^{y,y}_{1,2} 
\Big( 
t^{1}_{0} < T^{1}, \gamma^0 [u_{1}, u_{2} ]  \cap S^2_+ = \emptyset, \  S^2_+ \cap B (\hat{x} , s) = \emptyset  
\Big) 
\asymp s^{-1+ \alpha} m^{-2},
\end{equation}
where we have used $\alpha < 1 $ to simplify the exponents.
For each term in~\eqref{mendoi},
\begin{align*}
{\rm I}:=\;&P^{y,y}_{1,2} \Big( t^{1}_{0} < T^{1}, \ \gamma^0 [u_{1}, u_{2} ]  \cap S^2_+ = \emptyset, \  S^2_+ \cap B (\hat{x}, s) = \emptyset   \Big) \\
&=
P^{y,y}_{1,2} 
\left( 
\begin{array}{c}
t^{1}_{0} < T^{1},  \ u_{1} = u_{1}',   \ \gamma^0 [u_{1}, u_{2} ]  \cap S^2_+ = \emptyset, \  S^2_+ \cap B (\hat{x} , s) = \emptyset 
\end{array}
\right) \\
& \ \ \ + 
P^{y,y}_{1,2} 
\left( 
\begin{array}{c}
t^{1}_{0} < T^{1},  \ u_{1} \neq u_{1}', \ \gamma^0 [u_{1}, u_{2} ]  \cap S^2_+ = \emptyset, \  S^2_+ \cap B (\hat{x} , s) = \emptyset
\end{array}
\right)\\
&=:{\rm II}+{\rm III}.
\end{align*}
By \eqref{eq:bound1} and \eqref{eq:bound2}, one has that
$$
{\rm III}=O \big(s^{\delta_{0}}\big)\cdot{\rm I},
$$
while it is easy to see that 
$$
{\rm II}\leq P^{y,y}_{1,2} 
\left( 
\begin{array}{c}
t^{1}_{0} < T^{1},  \ \gamma^0 [u_{1}', u_{2} ]  \cap S^2_+ = \emptyset, \  S^2_+ \cap B (\hat{x} , s) = \emptyset 
\end{array}
\right)\leq {\rm I},
$$
which implies
\begin{align*}
&P^{y,y}_{1,2} \Big( t^{1}_{0} < T^{1}, \ \gamma^0 [u_{1}', u_{2} ]  \cap S^2_+ = \emptyset, \  S^2_+ \cap B (\hat{x} , s) = \emptyset   \Big) \\
&= P^{y,y}_{1,2} \Big( t^{1}_{0} < T^{1}, \ \gamma^0 [u_{1}, u_{2} ]  \cap S^2_+ = \emptyset, \  S^2_+ \cap B (\hat{x} , s) = \emptyset   \Big) 
 \Big[ 1 + O \big(s^{\delta_{0}}\big) \Big].
\end{align*}
Thus, we can replace $u_{1}$ by $u_{1}'$, which is easier to analyze. Indeed, if we let 
\begin{equation}\label{mou}
t_{2}  = \inf \Big\{ j \ge 0 \ \Big| \ S^{2} (j) \notin B \big(\hat{x} , s^{\delta_{0}} \big) \Big\}
\end{equation}
be the first time that $S^{2}$ leaves the ball  $B \big( \hat{x} , s^{\delta_{0}} \big) $, and write 
\[
    S^2_-=S^{2} [ 1, t_{2} ]
\]
for short. Then, by the coupling argument as stated in Lemma 3.3 of \cite{Natural} (see also Lemma 5.4 of \cite{Natural} for an analogous estimate concerning one-point function),  it follows that for all $y, z \in \partial_{i} B(\hat{x}, s ) $
\begin{align}
&\frac{P^{y,y}_{1,2} \Big( t^{1}_{0} < T^{1}, \ \gamma^0 [u_{1}', u_{2} ]  \cap S^2_+ = \emptyset, \  S^2_+ \cap B (\hat{x} , s) = \emptyset   \Big)}{P^{z,z}_{1,2} \Big( t^{1}_{0} < T^{1}, \ \gamma^0 [u_{1}', u_{2} ]  \cap S^2_+ = \emptyset, \  S^2_+ \cap B (\hat{x} , s) = \emptyset   \Big)}  \nonumber  \\
&\qquad=
\frac{
P^{y,y}_{1,2} \Big( 
t^{1}_{0} < T^{1}, \ \gamma^0 [u_{1}', v_1 ]  \cap S^2_- = \emptyset,   S^2_- \cap B (\hat{x} , s) = \emptyset   
\Big)
}
{
P^{z,z}_{1,2} 
\Big(
 t^{1}_{0} < T^{1}, \ \gamma^0 [u_{1}',  v_1 ]  \cap S^2_- = \emptyset,   S^2_- \cap B (\hat{x} , s) = \emptyset   
\Big)
} 
\Big[ 1 + O \big( s^{ \delta_0} \big) \Big].
\label{eq:comp1}
\end{align}
However, it is easy to see that for all $y, z \in \partial_{i} B(\hat{x}, s ) $
\begin{equation*}
P^{y} \big( t^{1}_{0} < T^{1} \big) = P^{z} \big( t^{1}_{0} < T^{1} \big) \Big[ 1 + O \big(\sqrt{s}\big) \Big]. 
\end{equation*}

Recall that $\gamma^0$ denotes the loop-erasure of a simple random walk starting at $y \in B(\hat{x}, s)$ and stopped upon hitting the origin $0$. The segment $\gamma^0 [u_1', u_2]$ corresponds to the portion of this LERW between its last visit to the boundary of $ B (\hat{x}, s)$ (before exiting $ B (\hat{x}, s^{\delta_0})$) and the origin. In~\eqref{eq:comp1}, we compared $\gamma^0 [u_1', u_2]$  with a shorter segment,  $\gamma^0 [u_1', v_1]$, which is stopped upon leaving $B(\hat{x}, s^{\delta^0})$ for the first time.
We now proceed to compare $\gamma^0 [u_1', v_1]$ with a corresponding segment of an infinite LERW, which is more amenable to analysis. This comparison also extends to $\gamma^0 [u_1', u_2]$. 

We write
\[
\gamma^{\infty} = \text{LE} \big( S^{1} [0, \infty ) \big)
\]
for the infinite LERW in $m^{-1} \mathbb{Z}^{3}$. If we let 
\begin{align}\label{yada}
&b = \inf \Big\{ j \ge 0 \ \Big| \ \gamma^\infty (j) \notin B \big( \hat{x} , s^{\delta_{0}} \big) \Big\} \text{ and } \notag \\
&a =  \max \Big\{ j \in [0, b ] \ \Big| \  \gamma^\infty (j) \in B (\hat{x} , s)  \Big\},
\end{align}
then it follows from~\eqref{eq:comp1} and~\cite[Proposition 4.4]{Masson} that for all $y, z \in \partial_{i} B(\hat{x}, s ) $
\begin{align}\label{mendoi-1}
&\frac{P^{y,y}_{1,2} \Big( t^{1}_{0} < T^{1}, \ \gamma^0 [u_{1}', u_{2} ]  \cap S^2_+ = \emptyset, \  S^2_+ \cap B (\hat{x} , s) 
= \emptyset   \Big)}{P^{z,z}_{1,2} \Big( t^{1}_{0} < T^{1}, \ \gamma^0 [u_{1}', u_{2} ]  \cap S^2_+ = \emptyset, \  S^2_+ \cap B (\hat{x} , s) = \emptyset   \Big)} \notag  \\
&\qquad=\frac{P^{y,y}_{1,2} \Big(  \gamma^\infty [a, b ]  \cap S^2_- = \emptyset, \  S^2_- \cap B (\hat{x} , s) = \emptyset   \Big)}{P^{z,z}_{1,2} \Big(  \gamma^\infty [a, b]  \cap S^2_- = \emptyset, \  S^2_-  \cap B (\hat{x} , s) = \emptyset   \Big)} \Big[ 1 + O \big(s^{\delta_{1}}\big) \Big].
\end{align}

For $L > 1$, we let 
\begin{equation*}
t_{2, L} = \inf \Big\{ j \ge 0 \ \Big| \ S^{2} (j) \notin B \Big( \hat{x}, s + s^{L} \Big) \Big\}.
\end{equation*}
By classical escape probability estimates for SRW, as in~\cite[Proposition 1.5.10]{Lawb}, and the asymptotics of non-intersection probabilities of LERW and SRW from Theorem~\ref{thm:esnew}, one has that 
\begin{align*}
&P^{y} 
\Big( 
  \text{diam} \big( S^{2} [0, t_{2, L} ] \big) \ge s^{L/2},  S^2_+ \cap B (\hat{x} , s) = \emptyset  
\Big) 
\le C m^{-1} s^{-L} O \big(  s^{L/2} \big) \\
&= 
P^{y,y}_{1,2} 
\Big(  
  \gamma^\infty [a, b ]  \cap S^2_- = \emptyset, S^2_-  \cap B (\hat{x} , s) = \emptyset  
\Big) 
O \Big(  s^{-L/4}  \Big).
\end{align*}
Therefore, we have 
\begin{align*}
P^{y,y}_{1,2} 
&\left(   
\gamma^\infty [a, b ]  \cap S^2_- = \emptyset, S^2_-  \cap B (\hat{x} , s) = \emptyset, \  \text{diam} \big( S^{2} [0, t_{2, L} ] \big) \le s^{L/2} 
\right) \\
&= P^{y,y}_{1,2} \Big(  \gamma^\infty [a, b ]  \cap S^2_- = \emptyset, \  S^2_-  \cap B (\hat{x} , s) = \emptyset   \Big) 
 \Big[ 1 + O \Big(  s^{-L/4}  \Big) \Big].
\end{align*}

We next want to replace the event 
\begin{equation*}
\Big\{ \gamma^\infty [a, b ]  \cap S^2_- = \emptyset \Big\} \qquad\text{ by }\qquad \Big\{ \gamma^\infty [a, b ]  \cap S^{2} [ t_{2, L}, t_{2} ] = \emptyset \Big\}.
\end{equation*}
To do it, we first claim that by \eqref{darui-6} 
\begin{equation*}
P^{y} \Big( \gamma^\infty [0, b ]  \cap B \big( \hat{x}, s- s^{L/4} \big) = \emptyset \Big) \le C s^{\delta_{2} L },
\end{equation*}
for some universal constant $\delta_{2} > 0$. This implies that taking $L$ sufficiently large 
\begin{align*}
&P^{y,y}_{1,2} \Big(  \gamma^\infty [a, b ]  \cap S^2_- = \emptyset, \  S^2_-  \cap B (\hat{x} , s) = \emptyset,    \gamma^\infty [0, b ]  \cap B \big( \hat{x}, s- s^{L/4} \big) \neq \emptyset   \Big) \\
&= P^{y,y}_{1,2} \Big(  \gamma^\infty [a, b ]  \cap S^2_- = \emptyset, \  S^2_-  \cap B (\hat{x} , s) = \emptyset   \Big) \Big[ 1 + O \big(s\big) \Big].
\end{align*}
Now suppose that both 
\begin{equation*}
\left\{  
  \begin{array}{c}
  \gamma^\infty [a, b ]  \cap S^{2} [ 1, t_{2, L} ] \neq \emptyset, \\  S^2_-  \cap B (\hat{x} , s) = \emptyset, \  \text{diam} \big( S^{2} [0, t_{2, L} ] \big) \le s^{L/2}
  \end{array} 
\right\}
\mbox{ and }
\Big\{ \gamma^\infty [0, b ]  \cap B \big( \hat{x}, s- s^{L/4} \big) \neq \emptyset  \Big\}
\end{equation*}
occur. This implies both of the following two events occur:  
\begin{itemize}
\item $S^{1}$ returns to $B \big( y, s^{L/2} \big) $ after it leaves $B \big( y, s^{L/4} \big) $,

\item $S^{2} [1, t_{2} ]  \cap B (\hat{x} , s) = \emptyset$.

\end{itemize}
By classical estimates on escape probabilities for SRW, the probability of these two events occurring is bounded above by $C m^{-1} s^{\frac{L}{4} -1}$, which is, taking $L$ sufficiently large, equal to 
\begin{equation*}
P^{y,y}_{1,2} \Big(  \gamma^\infty [a, b ]  \cap S^2_- = \emptyset, \  S^2_-  \cap B (\hat{x} , s) = \emptyset   \Big) O \big(s\big).
\end{equation*}
Combining these estimates, for sufficiently large $L$, we see that 
\begin{align*}
&P^{y,y}_{1,2} \Big(  \gamma^{\infty} [a, b ]  \cap S^{2} [ t_{2, L}, t_{2} ] = \emptyset, \ S^2_-  \cap B (\hat{x} , s) = \emptyset   \Big) \\
&= P^{y,y}_{1,2} \Big(  \gamma^{\infty} [a, b ]  \cap S^2_- = \emptyset, \  S^2_-  \cap B (\hat{x} , s) = \emptyset   \Big) \Big[ 1 + O \big(s\big) \Big].
\end{align*}

Now we can couple $\gamma^{\infty}$ and its scaling limit ${\eta}_{\infty}$. We also couple $S^{2}$ and the Brownian motion $W$.  Taking $m$ and $L$ sufficiently large, we see that 
\begin{align*}
&P^{y,y}_{1,2} \Big(  \gamma^{\infty} [a, b ]  \cap S^{2} [ t_{2, L}, t_{2} ] = \emptyset, \  S^2_-  \cap B (\hat{x} , s) = \emptyset   \Big) \\
&=\Big[ 1 + O \big(s\big)  \Big] \times  \sum_{w \in \partial B \big( \hat{x}, s + s^{L} \big) \cap B (y, s^{L/2}) } P^{y} \Big( S^{2} [1, t_{2, L} ] \cap  B (\hat{x} , s) = \emptyset,  \ S^{2} (t_{2, L} ) = w \Big)  \\
& \ \ \ \ \ \ \ \ \ \ \ \ \  \ \ \ \ \ \ \ \times P^{y, w}_{1,2} \Big( \gamma^{\infty} [a, b ]  \cap S^{2} [ 0, t_{2} ] = \emptyset, \  S^{2} [0, t_{2} ]  \cap B (\hat{x} , s) = \emptyset   \Big) \\
&= \Big[ 1 + O \big(s\big)  \Big]  P^{y} \Big( S^{2} [1, t_{2} ] \cap  B (\hat{x} , s) = \emptyset \Big) \\
& \ \ \ \ \ \ \ \ \ \ \ \ \  \ \ \ \ \ \ \times P^{y, w_{y}}_{1,2} \Big( \gamma^{\infty} [a, b ]  \cap S^{2} [ 0, t_{2} ] = \emptyset \ \Big| \  S^{2} [0, t_{2} ]  \cap B (\hat{x} , s) = \emptyset   \Big) \\
&= \Big[ 1 + O (s)  \Big]  P^{y} \Big( S^{2} [1, t_{2} ] \cap  B (\hat{x}, s) = \emptyset \Big)  \\  
& \ \ \ \ \ \ \ \ \ \ \ \ \  \ \ \ \ \ \ \times P^{y_{0}, w_{y_{0}}}_{1,2} \Big( \gamma^{\infty} [a, b ]  \cap S^{2} [ 0, t_{2} ] = \emptyset \ \Big| \  S^{2} [0, t_{2} ]  \cap B (\hat{x} , s) = \emptyset   \Big),
\end{align*}
where in the second equality, $w_{y}$ is the unique point in $ \partial B \big(\hat{x}, s + s^{L} \big) \cap B (y, s^{L/2})$ such that $\hat{x}$, $y$ and $w_{y}$ are lying on the same straight line and in the third equality,  $y_{0} \in \partial_{i} B(\hat{x}, s ) $ is the ``north pole'' of $B(\hat{x}, s )$. Note that the third equality follows from approximation of non-intersection probabilities by the continuous counterpart (similar to the argument in \cite[Proposition 3.16]{Escape}) and the invariance under rotation of $\eta_\infty$ and $W$.

Therefore, returning to \eqref{mendoi-1}, it follows that for all $y, z \in \partial_{i} B(\hat{x}, s ) $,
\[
\begin{split}
\frac{P^{y,y}_{1,2} \Big( t^{1}_{0} < T^{1}, \ \gamma^0 [u_{1}, u_{2} ]  \cap S^2_+ = \emptyset, \  S^2_+ \cap B (\hat{x}, s) = \emptyset   \Big)}{P^{z,z}_{1,2} \Big( t^{1}_{0} < T^{1}, \ \gamma^0 [u_{1}, u_{2} ]  \cap S^2_+ = \emptyset, \  S^2_+ \cap B (\hat{x} , s) = \emptyset   \Big)} \\ = \Big[ 1 + O \big( s^{c_{0}} \big) \Big] \frac{P^{y} \Big( S^2_- \cap  B (\hat{x}, s) = \emptyset \Big)}{P^{z} \Big( S^2_- \cap  B (\hat{x} , s) = \emptyset \Big)}.
\end{split}
\]

This gives that 
\begin{align*}
P \Big( \gamma_{m} \cap B(\hat{x},s) \neq \emptyset \Big)&=  P^{y_{0},y_{0}}_{1,2} \Big( t^{1}_{0} < T^{1}, \ \gamma^0 [u_{1}, u_{2} ]  \cap S^2_+ = \emptyset, \  S^2_+ \cap B (\hat{x} , s) = \emptyset   \Big) \\
& \qquad \qquad \times \Big[ 1 + O \big( s^{c_{0}} \big) \Big]  \sum_{y \in \partial_{i} B(\hat{x}, s ) } 
  \frac{P^{y} \Big( S^2_- \cap  B (\hat{x} , s) = \emptyset \Big)}{P^{y_{0}} \Big( S^2_- \cap  B (\hat{x} , s) = \emptyset \Big)}.
\end{align*}

It is known (see Subsection 2.2 of \cite{Lawb}) that 
\begin{equation*}
 \sum_{y \in \partial_{i} B(\hat{x}, s ) } \frac{P^{y} \Big( S^2_- \cap  B (\hat{x}, s) = \emptyset \Big)}{P^{y_{0}} \Big( S^2_- \cap  B (\hat{x} , s) = \emptyset \Big)} = c_{1} s^{2} m^{2}  \Big[ 1 + O \big(s^{c_{2}}\big) \Big].
 \end{equation*}
Thus, it suffices to obtain a sharp estimate on 
 \begin{equation*}
 P^{y_{0},y_{0}}_{1,2} \Big( t^{1}_{0} < T^{1}, \ \gamma^0 [u_{1}, u_{2} ]  \cap S^2_+ = \emptyset, \  S^2_+ \cap B (\hat{x}, s) = \emptyset   \Big).
\end{equation*}
 But, by classical hitting probability estimates and asymptotics for Green's function, it is not difficult to show that for sufficiently large $m$
 \begin{equation*}
 P^{y_{0},y_{0}}_{1,2} \Big( t^{1}_{0} < T^{1}, \   S^2_+ \cap B (\hat{x}, s) = \emptyset   \Big) = q_{1} s^{-1} m^{-2} \Big[ 1 + O \big(s^{q_{2}}\big) \Big],
 \end{equation*}
 for some universal constants $q_{1}, q_{2} > 0$. Thus, it suffices for us to give a sharp estimate on 
 \begin{equation}\label{mendoi-2.1}
 p(s)= p(s, m) \coloneqq P^{y_{0},y_{0}}_{1,2} \Big(  \text{LE} \big( X^{1} [0, t^{1}_{0}] \big) [u_{1}, u_{2} ]  \cap X^{2} [ 1, T^{2} ] = \emptyset   \Big),
 \end{equation}
 where $X^{1}$ is $S^{1}$ conditioned that $t^{1}_{0} < T^{1}$ and $X^{2}$ is $S^{2}$ conditioned that $S^2_+ \cap B (\hat{x} , s) = \emptyset$. 
We obtain such estimate in Lemma~\ref{lemma:non-intersection}, and this completes the proof of the proposition.
\end{proof}

We finish this subsection with the following lemma for the random walks $S^1$ and $S^2$ conditioned to the events appearing in~\eqref{mendoi-2.1}.
Let $X^1$ be the simple random walk $S^{1}$ conditioned to $t^{1}_{0} < T^{1}$, i.e. $X^1$ hits the origin before $\mathbb{D}$. 
Let $X^2$ be the simple random walk $S^{2}$ conditioned that $S^2_+ \cap B (\hat{x} , s) = \emptyset$, then $X^2$ does not return to the ball $B (\hat{x}, s)$. We show next the sharp estimate on their intersection probability.

\begin{lemma} \label{lemma:non-intersection}
 Let $X^{1}$ and $X^{2}$ be the conditioned simple random walks above. 
 There exists $m_0 > 0$ such that, for $m > m_0$ there exists universal constants $q_{3}, q_{4} > 0$ satisfying
 \begin{equation}\label{mendoi-2}
  P^{y_{0},y_{0}}_{1,2} \Big(  {\rm LE} \big( X^{1} [0, t^{1}_{0}] \big) [u_{1}, u_{2} ]  \cap X^{2} [ 1, T^{2} ] = \emptyset   \Big) = q_{3} s^{\alpha } \Big[ 1 + O (s^{q_{4}} ) \Big].
  \end{equation}
\end{lemma}

\begin{proof}[Proof of Lemma~\ref{lemma:non-intersection}]
Recall the definitions of and $u_1$ and $u_2$ in \eqref{eq:u1} and \eqref{eq:u2}, respectively.
Let us write
\[
  p(s)= p(s, m) \coloneqq P^{y_{0},y_{0}}_{1,2} \Big(  \text{LE} \big( X^{1} [0, t^{1}_{0}] \big) [u_{1}, u_{2} ]  \cap X^{2} [ 1, T^{2} ] = \emptyset   \Big).
\]
We first deal with the case that $s= 2^{-k}$.  Comparison results between the loop-erasure of $X^{1}$ and an infinite LERW (see Propositions 4.2, 4.4 and Corollary 4.5 of \cite{Masson} for this) show that, for an infinite LERW in $m^{-1} \mathbb{Z}^{3}$, again denoted by $\gamma^{\infty} = \text{LE} \big( S^{1} [0, \infty ) \big)$,
\begin{equation*}
\frac{p(2^{-k-1})}{p(2^{-k})} = \frac{P^{y_{0},y_{0}}_{1,2} \Big(  \gamma^{\infty} [a,  b]  \cap X^{2} [ 1, t_{2} ] = \emptyset   \Big)}{P^{y_{0}',y_{0}'}_{1,2} \Big(  \gamma^{\infty} [a',  b]  \cap X^{2} [ 1, t_{2} ] = \emptyset   \Big)} \ \Big[ 1 + O \big( 2^{-ck} \big) \Big],
\end{equation*}
where 
\begin{itemize}
\item $\delta_0  = (1-\alpha) /  2$, as defined in~\eqref{eq:delta0},

\item $y_{0}$ is the north pole of $B (\hat{x}, 2^{-k-1} )$ and $y_{0}'$ is the north pole of $B (\hat{x}, 2^{-k} )$

\item $b = \inf \Big\{ j \ge 0 \ \Big| \ \gamma^{\infty} (j) \notin B \big( \hat{x} , 2^{- \delta_{0} k} \big) \Big\}$,

\item $a=  \max \Big\{ j \in [0, b ] \ \Big| \   \gamma^{\infty} (j) \in B (\hat{x}, 2^{-k-1})  \Big\}$,

\item $a' = \max \Big\{ j \in [0, b ] \ \Big| \   \gamma^{\infty} (j) \in B (\hat{x}, 2^{-k})  \Big\}$ and

\item $t_{2} = \inf \Big\{ j \ge 0 \ \Big| \ X^{2} (j) \notin B \big( \hat{x}, 2^{- \delta_{0} k} \big) \Big\}$.

\end{itemize}
The same comparison argument gives that  
\begin{equation*}
\frac{p(2^{-k})}{p(2^{-k+1})} = \frac{P^{y_{0}',y_{0}'}_{1,2} \Big(  \gamma^{\infty} [a'',  b']  \cap X^{2} [ 1, t_{2}' ] = \emptyset   \Big)}{P^{y_{0}'',y_{0}''}_{1,2} \Big(  \gamma^{\infty} [a''',  b']  \cap X^{2} [ 1, t_{2}' ] = \emptyset   \Big)} \ \Big[ 1 + O \big( 2^{-ck} \big) \Big],
\end{equation*}
where we let 
\begin{itemize}

\item $y''_{0}$ is the north pole of $ B \big( x_{0} , 2^{- \delta_{0} k-1} \big) $,

\item $b' = \inf \Big\{ j \ge 0 \ \Big| \ \gamma^{\infty} (j) \notin B \big( \hat{x}, 2^{- \delta_{0} k -1} \big) \Big\}$,

\item $a''=  \max \Big\{ j \in [0, b' ] \ \Big| \   \gamma^{\infty} (j) \in B (\hat{x} , 2^{-k})  \Big\}$,

\item $a''' = \max \Big\{ j \in [0, b ] \ \Big| \   \gamma^{\infty} (j) \in B (\hat{x}, 2^{-k+1})  \Big\}$ and

\item $t_{2}' = \inf \Big\{ j \ge 0 \ \Big| \ X^{2} (j) \notin B \big( \hat{x}, 2^{- \delta_{0} k-1} \big) \Big\}$.

\end{itemize}
Replacing $\gamma^{\infty}$ and $X^{2}$ by their scaling limits which satisfy the scale invariance, we see that 
\begin{equation*}
 \frac{p(2^{-k-1})}{p(2^{-k})} = \frac{p(2^{-k})}{p(2^{-k+1})} \Big[ 1 + O \big( 2^{-ck} \big) \Big].
\end{equation*}
This ensures that 
\begin{equation}\label{mendoi-3}
p (2^{-k} ) = c^{0} 2^{- \alpha k } \Big[ 1 + O \big(2^{-c k } \big) \Big] \text{ as } k \to \infty,
\end{equation}
for some universal constants $c^{0}, c > 0$.

Similarly, for each $r \ge 0 $, one can prove that there exists $f(r) > 0$ depending on $r$ and universal $c > 0$ such that 
\begin{equation}\label{mendoi-4}
p (2^{- r} 2^{-k} ) = f(r) c^{0} 2^{- \alpha k } \Big[ 1 + O \big(2^{-c k } \big) \Big] \text{ as } k \to \infty.
\end{equation}
Note that $f(0) = 1$ and $f(1) = 2^{- \alpha}$. From now on, we will show that $f(r) = 2^{- \alpha r }$. To do this, we will first prove that 
\begin{equation}\label{mendoi-5}
f ( r_{1} + r_{2} ) = f(r_{1} ) f(r_{2} ) \  \ \  \text{ for all } \  r_{1}, r_{2} \ge 0.
\end{equation}
We can again use the scale invariance of the scaling limits of $\gamma^{\infty}$ and $X^{2}$ to show that 
\begin{equation*}
\frac{p ( 2^{-k - r_{2}} )}{p ( 2^{-k} )} = \frac{p ( 2^{- k - r_{1}})}{p ( 2^{-k-r_{1} -r_{2}} )} \Big[ 1 + O \big( 2^{-ck} \big) \Big].
\end{equation*}
This implies \eqref{mendoi-5}. Furthermore, using the coupling argument as in Step 2 of the proof of Proposition~\ref{prop:continuity}, it is not difficult (but rather tedious)  to prove that $f$ is continuous. Combining these properties with $f(0) =1$ and $f(1) = 2^{-\alpha}$. it holds that $f ( r) = 2^{- \alpha r }$. Consequently, we have 
\begin{equation}\label{mendoi-6}
p (2^{- (k+r)} ) =  c^{0} 2^{- \alpha (k +r) } \Big[ 1 + O \big(2^{-c (k+r) } \big) \Big] \text{ as } k \to \infty,
\end{equation}
for all $r \in [0,1]$. Equivalently, we have $p (s) = c^{0} s^{\alpha} \Big[ 1 + O \big(s^{c} \big) \big]$. This finishes the proof of the lemma.
\end{proof}

\subsection{Asymptotics of the one-point function}

We now turn to the asymptotics of the one-point function $g(\cdot)$ and prove Corollary~\ref{cor:asymptotics}.

\begin{proof}[Proof of Corollary~\ref{cor:asymptotics}]
   Recall that $\hat{x} = (1/2, 0 , 0)$ serves as our reference point.
  From~\eqref{eq:one-pointb0} we have that for $n \geq 2$
  \[
    a_1 \vert \hat{x} \vert^{- (3 - \beta)} \leq g (n^{-1} \hat{x} )  n^{\beta - 3} \leq a_2 \vert \hat{x} \vert^{- (3 - \beta)} .
  \]
  By continuity and rotational invariance of the one-point function (Theorem~\ref{thm:oneptnew}), in order to prove~\eqref{eq:asymp0} it suffices to verify that for some $b_1 > 0$
  \begin{equation} \label{eq:const-conv}
  g (n^{-1} \hat{x}) n^{\beta - 3} = b_1 \Big( 1 + O \big( n^{-\delta_1} \big) \Big) \qquad \text{ as } n \to \infty.
  \end{equation}

  Let $x \in \mathbb{D} \setminus \{ 0 \}$ be such that $\vert x \vert < 10^{-1}$.  
  From Theorem~\ref{thm:oneballnew}, for  we have that for any $ 0 < r < 10^{-1}d_{x} $
  \begin{equation} \label{eq:g-cont}
   \frac{g(2x)}{g(x)} 
   = \frac{ P ( \KK \cap B(2x, r) \neq \emptyset ) }{ P ( \KK \cap B(x, r) \neq \emptyset ) } 
    \Big[ 1 + O\big(d_x^{-c}r^{\delta}\big) \Big],
  \end{equation}
  for $\delta > 0$ as as in Theorem~\ref{thm:oneballnew}.

  Recall that $\DD$ denotes the open unit ball. For each $m \in \R^+$,  $\eta_m$ is the LERW on $\DD_m$. Keeping the notation introduced in Subsection~\ref{se:scalinglimit}, 
  let $2 \DD$ denote the open ball of radius 2, and define its discretization as $2 \DD_m = m^{-1} \mathbb{Z}^3 \cap 2\DD$. 
  
  Under this setup,  
  $\eta^{2}_{m} = \eta_{m/2} $ defines a LERW on $2 \DD_m$,  and $\KK^{2} $ is the scaling limit of  $ \eta^{2}_{m} $ on $2 \DD_m$ as $m \to \infty$. Note that under this scaling, $\KK^{2}$ is a simple curve on $2\DD$.
  Then, for any $m \in \R^+$, we have 
  \begin{equation} \label{eq:g-scaling}
    \frac{ P \left( (4x)_m \in \eta_m  \right) }{ P \left( (2x)_m \in \eta_m  \right) }
    =
    \frac{ P \left( (4x)_{m/2} \in \eta^2_{m}  \right) }{ P \left( (2x)_{m/2} \in \eta^2_{m}  \right) },
  \end{equation}
  where the right-hand side of this equation refers to one-point probabilities on $2\DD_m$.
  
  Taking $\delta > 0$ as in~\eqref{eq:g-cont},  let $m \in \R^+$ be such that $m^{-1} < r^{2\delta} $.  Then, by Theorem~\ref{thm:oneballnew}, there exist constants $c', \delta'>0$ such that
  \begin{equation} \label{eq:g-scalinglimit}
    \frac{ P \left( (4x)_{m/2} \in \eta^2_{m}  \right) }{ P \left( (2x)_{m/2} \in \eta^2_{m}  \right) } = 
    \frac{ P ( \KK^{2} \cap B(4x, 2r) \neq \emptyset ) }{ P ( \KK^{2} \cap B(2x, 2r) \neq \emptyset ) }  \Big[ 1 + O \big( d_x^{-c'} r^{\delta'} \big) \Big]. 
  \end{equation}
  The scaling invariance of the scaling limit of the LERW implies 
  \begin{equation} \label{eq:g-scalingInv}
  \frac{ P ( \KK^{2} \cap B(4x, 2r) \neq \emptyset ) }{ P ( \KK^{2} \cap B(2x, 2r) \neq \emptyset ) } 
  =
  \frac{ P ( \KK \cap B(2x, r) \neq \emptyset ) }{ P ( \KK \cap B(x, r) \neq \emptyset ) }.
  \end{equation} 
  Therefore, from~\eqref{eq:g-cont},~\eqref{eq:g-scaling},~\eqref{eq:g-scalinglimit} and~\eqref{eq:g-scalingInv} we have that there exist constants $ c'', \delta'' > 0 $
  \begin{equation} \label{eq:g-eq}
    \frac{g(2x)}{g(x)}  = \frac{ P \left( (4x)_m \in \eta_m  \right) }{ P \left( (2x)_m \in \eta_m  \right) } \Big[ 1 + O\big(d_x^{-c''}r^{\delta''}\big) \Big].
  \end{equation}
  On the other hand, Theorem~\ref{thm:oneptnew} implies 
  \begin{equation} \label{eq:g-def}
    \frac{g(4x)}{g(2x)} = \frac{P \left( (4x)_m \in \eta_m  \right) }{  P \left( (2x)_m \in \eta_m  \right)  } \Big[ 1 + O\big(d_x^{-c}m^{-\delta}\big) \Big].
  \end{equation}
  It follows from~\eqref{eq:g-eq} and~\eqref{eq:g-def} that
  $  \frac{g(2x)}{g(x)}  = \frac{g(4x)}{g(2x)}
    $.
  Note that the argument above relies on the scaling invariance of $\KK$. Then, replacing the blow-up by a factor of $2$ with a blow-up by $s > 0$, we get that
  \[
    \frac{g(2x)}{g(x)}  = \frac{g(2sx)}{g(sx)} \qquad \text{for all }s > 0.
  \]
 On the other hand, a similar argument used to prove  \eqref{eq:funceq} gives that for each $r \in [0,1]$
 $$ \frac{g \big( 2^{-r} \cdot 2^{-n} \hat{x}\big)}{g (2^{-n} \hat{x})} = 2^{r (3-\beta)} \Big[ 1+ O (2^{-\delta n}) \Big], \ \ \ \ \text{as} \  \ n \to  \infty.$$
  Combining these estimates, we obtain~\eqref{eq:const-conv} and hence~\eqref{eq:asymp0}.

  Given~\eqref{eq:one-pointb1}, a similar argument holds for the asymptotic~\eqref{eq:asymp1} as $\vert x \vert \rightarrow 1$. 
\end{proof}

\section{Two-point estimates} \label{se:twopoint}

This section focuses on estimates related to the two-point hitting probability.
First, we present an $L^2$-estimate in Proposition~\ref{prop:3.0}, which depends on the off-diagonal two-point estimate in Proposition~\ref{prop:crux2pt}. In Subsection~\ref{sec:two-proof}, we proceed to prove Theorem~\ref{thm:two-point}.

\subsection{An \texorpdfstring{$L^2$}{L2} estimate}

Consider a fixed $V\in\cD^o$ and recall Proposition \ref{prop:generalMC} (with $A$, $\delta$, $d$ therein as $\KK$, $\beta$, $3$ resp.) for 
\begin{equation}
J_s(z)\coloneqq 2^{(3-\beta)s}1_{z\in B(\KK,2^{-s})},
\end{equation}
the renormalized indicator function of the blow-up of $\KK$.  Write $$J_s(V)=\int_V J_s(z) dz$$ for the total $2^{-s}$-content of $\KK$ in $V$.

\begin{prop} \label{prop:3.0}
  Let $V\in \cD^o$ be a dyadic box.
  For any $s\geq 100$ and $s'\in [s,s+1]$, 
  \begin{equation}\label{EQ2}
  E\left[(J_s(V)-J_{s'}(V))^2\right]=O_V\left(2^{-c s^{1/2}}\right).
  \end{equation}
\end{prop}

\begin{proof}[Proof of Proposition \ref{prop:3.0}]

We decompose the integral in \eqref{EQ2} into three parts: off-diagonal, near-diagonal and on-diagonal. Writing 
\begin{equation} \label{eq:defK}
K=K(z,w,s,s')=\big(J_s(z)-J_{s'}(z)\big)\big(J_s(w)-J_{s'}(w)\big)
\end{equation}
for short, we have
\begin{equation}
\begin{split}
E\left[(J_s(V)-J_{s'}(V))^2\right]
=\;&\iint_{z,w\in V, |z-w|\geq 2^{-\sqrt{s}/100}} E[K]dzdw\\
+\;&\iint_{z,w\in V, 100\cdot 2^{-s}\leq|z-w|\leq 2^{-\sqrt{s}/100}} E[K]dzdw\\
+\;&\iint_{z,w\in V, |z-w|\leq 100\cdot 2^{-s}} E[K]dzdw\\
\coloneqq \;&{\rm I}+{\rm II}+{\rm III}.
\end{split}
\end{equation}

We deal with ${\rm I}$, the crucial off-diagonal integral, with a two-point estimate in Proposition~\ref{prop:crux2pt}.
We postpone the proof of this proposition to Subsection~\ref{se:cruxproof} and turn to the near-diagonal term ${\rm II}$. It follows from Proposition \ref{prop:uptocb2}, the up-to-constant two-point estimates that (recall that $s'\in [s,s+1]$)
\begin{equation}\label{eq:II}
\begin{split}
\big|{\rm II}\big|\;&\leq \iint_{z,w\in V, 100\cdot 2^{-s}\leq|z-w|\leq 2^{-\sqrt{s}/100}} E\Big[\big(J_s(z)+J_{s'}(z)\big)\big(J_s(w)+J_{s'}(w)\big)\Big]dzdw\\
&\leq c \int_V \int_{100\cdot 2^{-s}}^{2^{-\sqrt{s}}} \int_{\partial B(z,r)} \left(2^{(3-\beta)s}\right)^2 \cdot\left(\frac{2^{-2s}}{r}\right)^{3-\beta} \sigma_r(dw)drdz \\& \leq C_V 2^{-(\beta-1)\sqrt{s}/100},
\end{split}
\end{equation}
where $\sigma_r(\cdot)$ is the surface measure on $\partial B(z,r)$. Finally, the on-diagonal term ${\rm III}$ can be controlled easily by a naive upper bound $J_{s}(w)\leq 2^{(3-\beta)s}$ and the one-point estimate Proposition \ref{prop:uptocb1}. We obtain that:
\begin{equation}\label{eq:III}
\big|{\rm III}\big| \leq \int_V \int_{B(z,100\cdot 2^{-s})} E\big[J_s(z)\big] 2^{(3-\beta)s} dwdz\leq C_V 2^{-\beta s}.
\end{equation}
The inequality \eqref{EQ2} now follows from combining \eqref{EQ1}, \eqref{eq:II} and \eqref{eq:III}.\end{proof}

\subsection{Off-diagonal two-point estimate} \label{se:cruxproof}

This subsection is dedicated to the proof of the crucial two-point estimate in Proposition~\ref{prop:3.0}.

\begin{prop}\label{prop:crux2pt}
  Let $V \in \cD^o$ be a dyadic box. 
  For any $s\geq 1$, any $z,w\in V$ such that $|z-w|\geq 2^{-\sqrt{s}/100}$, and all $s'\in[s,s+1]$, there exist $\delta,c_V>0$ such that
  \begin{equation}\label{EQ1}
    \left|E\big[K(z,w,s,s')\big]\right|\leq c_V 2^{-\delta \sqrt{s}}, 
  \end{equation}
  where $K(z,w,s,s')=\big(J_s(z)-J_{s'}(z)\big)\big(J_s(w)-J_{s'}(w)\big)$.
\end{prop}

As the understanding of the scaling limit of LERW is very limited and the essential tools are lacking in the continuum, rather than dealing with this two-point estimate directly in the continuum, we actually pass the analysis to the discrete.
Let
\[
  J_s^m(z)\coloneqq 2^{(3-\beta)s}1_{z\in B(\eta_m ,2^{-s})}
\]   
stand for the renormalized indicator function in the discrete. 
Proposition \ref{prop:continuity} gives us the following corollary on convergence of $J_s^m$ to $J_s$ in expectation.

\begin{cor} \label{cor:twoptapprox}
  For $z\in\DD$ and $s>0$ such that $B(z,2^{-s})\subset \DD$, it follows that
  \begin{equation} \label{eq:twoptapprox1}
  E\big[J^{m}_{s}(z)\big] \to E\big[J_{s}(z)\big]\mbox{ as $m\to\infty$}.
  \end{equation}
  Moreover, one has that for $z,w\in\DD$ and $s,s'>0$ such that $B(z,2^{-s})\cap B(w,2^{-s'})=\emptyset$ and $B(z,2^{-s}), B(w,2^{-s'})\subset \DD$, 
  \begin{equation} \label{eq:twoptapprox2}
  E\Big[J^{m}_{s}(z)J^{m}_{s'}(w)\Big] \to E\Big[J_{s}(z)J_{s'}(w)\Big]\mbox{ as $m\to\infty$}.
  \end{equation}
\end{cor}

\begin{proof}[Proof of Corollary~\ref{cor:twoptapprox}]
  This is a simple application of the fact that $\eta_m$ converges to $\KK$ in the $d_{\text{Haus}}$-metric.
\end{proof}

Thanks to Corollary~\ref{cor:twoptapprox}, we can conduct much of the analysis of $K(z,w,s,s')$ in the discrete. Indeed, by Corollary~\ref{cor:twoptapprox}, to prove Proposition~\ref{prop:crux2pt} it suffices to prove the following estimate for the renormalized indicator function in the discrete space. For the purpose of proving Proposition~\ref{prop:crux2pt}, it suffices to work with dyadic scales, since Proposition~\ref{prop:generalMC} only requires understanding the one and two point Green's function for a subsequence of scales tending to infinity.

\begin{prop}\label{prop:2preduction}
Let $V \in \cD^o$ be a dyadic box. Take $s \geq 1$, $s' \in [s, s+1]$ and let $z, w \in V$ satisfying $\vert z - w \vert \geq 2^{- \sqrt{s} / 100}$ (the assumption in Proposition \ref{prop:crux2pt}).  
Then there exists  $M_2(s)>0$, such that for all  $m=2^n\geq M_2$ for some $n\in\Zp$, one has
\begin{equation}\label{eq:2preduction}
\left|E\Big[\big(J^{m}_s(z)-J^{m}_{s'}(z)\big)\big(J^{m}_s(w)-J^{m}_{s'}(w)\big)\Big]\right|\leq O_{V}\left(2^{-\delta \sqrt{s}}\right).
\end{equation}
\end{prop}

\begin{proof}
Fix $x=(1/2,0,0)$ as the reference point. To prove \eqref{eq:2preduction}, it suffices to prove that for sufficiently large $m$, and for $S=s,s'$,
\begin{equation}\label{eq:1vs2}
\Big|E\left[J^{m}_{S}(z)\big(J^{m}_{s}(w)-J^{m}_{s'}(w)\big)\right]\Big|=O_V\left(2^{-\delta \sqrt{s}}\right).
\end{equation}
We will work with $S=s$ as the other case follows almost identically. 

Our strategy to control \eqref{eq:1vs2} is an application of \eqref{eq:bpbb}.
We first need to change $w$ to $w_m$, which is the closest point in $m^{-1} \mathbb{Z}^3$ to $w$. 
Let $s_1$, $s_0$ be such that
\[
  2^{-s_1}=2^{-s}-2^{-s^2}, \quad  2^{-s_0}=2^{-s}+2^{-s^2}
\]
and define $s'_0$, $s'_1$ similarly.
The Taylor expansion of  $\log (1 + x)$ around $x = 0$ shows that $s_1-s, s-s_0 \asymp 2^{s-s^2}\leq C2^{-s^2/2}$, hence
\begin{equation}\label{eq:asymp2power}
  \max\Big\{1-2^{-(3-\beta) (s_1-s)}, 2^{(3-\beta)(s-s_0)}-1\Big\} \leq C2^{-s^2/2}.
\end{equation}
Let $M_0 > 0$ be large enough so that, for $2^n \geq M_0$, Corollary~\ref{cor:asymp} gives bounds from the limit in~\eqref{eq:asymp} for the exponents  $s = s_0$, $s_1$, $s_1'$ and $s_0'$ (see \eqref{eq:cruxlast2} below).
Recall $M_1 > 0 $ from Proposition~\ref{prop:uptocb2} and set $M_2 \geq \max \{ 2^{s^2}, M_0, M_1 \}$. With this choice, for $m\geq M_2 \geq 2^{s^2}$,
\begin{equation}
2^{-(3-\beta) (s_1-s)}J^{m}_{s_1}(w_m) \leq J^{m}_{s}(w)\leq 2^{(3-\beta)(s-s_0)}   J^m_{s_0}(w_m).
\end{equation}
Therefore, with \eqref{eq:asymp2power}, 
\begin{equation}\label{eq:changew1}
\max_{i=0,1}\Big|E\left[J^{m}_{s}(z)\big(J^{m}_{s}(w)-J^{m}_{s_i}(w_m)\big)\right]\Big| \leq  C2^{-s^2/2} E\big[J^{m}_{s}(z)J^{m}_{s}(w_m)\big] .
\end{equation}
Similarly, 
\begin{equation}\label{eq:changew2}
  \max_{i=0,1}\Big|E\left[J^{m}_{s}(z)\big(J^{m}_{s'}(w)-J^{m}_{s'_i}(w_m)\big)\right]\Big| \leq  C2^{-s^2/2} E\big[J^{m}_{s}(z)J^{m}_{s'}(w_m)\big].
\end{equation}
Recall our assumption that  $|z-w|\geq 2^{-\sqrt{s}/100}$ and  $ |w-w_m| < 2^{-s^2} $ then $|z-w_m|\geq \frac{1}{10} 2^{-\sqrt{s}/50}$. 
By the up-to-constant two-point estimate \eqref{eq:2b}
\begin{equation}\label{eq:changew3}
E\big[J^{m}_{s}(z)J^{m}_{s}(w_m)\big] ,E\big[J^{m}_{s}(z)J^{m}_{s'}(w_m)\big] =O_V\left( 2^{(3-\beta)\sqrt{s}/50}\right).
\end{equation}
Returning to \eqref{eq:1vs2}, by the controls \eqref{eq:changew1}-\eqref{eq:changew3} above, there exists some $c>0$ such that
\begin{align}\label{eq:pingpong}
&\Big|E\left[J^{m}_{s}(z)\big(J^{m}_{s}(w)-J^{m}_{s'}(w)\big)\right]\Big| \notag \\ &\leq\max_{i=0,1}\Big|E\left[J^{m}_{s}(z)\big(J^{m}_{s_i}(w_m)-J^{m}_{s'_{1-i}}(w_m)\big)\right]\Big|+O_V \left(2^{-cs^2}\right).
\end{align}

We are now ready to apply \eqref{eq:bpbb} to the first term on the RHS of \eqref{eq:pingpong}: it follows that
\begin{equation}
E\left[J^{m}_{s}(z)J^{m}_{\tilde{s}}(w_m)\right]=2^{(3-\beta)(s+\widetilde{s})} \left(1+O\big(2^{-\sqrt{s}}\big)\right)A(\widetilde{s})Q,
\end{equation}
for $\widetilde{s}= s_0,s_1,s'_0,s'_1$, 
where we write (recall that $m=2^n$ and the definition of $a_n(\cdot)$ from \eqref{eq:anSdef})
$$
A(\widetilde{s})=a_n\left(B(2^{-\widetilde{s}})\cap 2^{-n}\Z^3\right)
\qquad
\mbox{and}
\qquad
Q= P\big(B(z,2^{-s})\cap \eta_m\neq\emptyset, w_m\in\eta_m\big)
$$
for short. By the assumption above that $ m  \geq M_2 \geq M_1 $ and the up-to-constant estimate on $Q$, see \eqref{eq:1b1p}, it follows that 
\begin{equation}
Q = O_V\big(2^{-(3-\beta)(s-\sqrt{s}/100+n)}\big).
\end{equation}
Hence  for $i=0,1$,
\begin{equation}\label{eq:cruxapplied}
\begin{split}
&\Big|E\left[J^{m}_{s}(z)\big(J^{m}_{s_i}(w_m)-J^{m}_{s'_{1-i}}(w_m)\big)\right]\Big|\\
\leq\;& 2^{(3-\beta)s}\left[ \Big(1+O\big(2^{-\sqrt{s}}\big)\Big)A(s_i)2^{(3-\beta)s_i}- \Big(1+O\big(2^{-\sqrt{s}}\big)\Big)A(s'_{1-i})2^{(3-\beta)s'_{1-i}} \right]Q\\
\leq\;&  \Big[{\rm I}     + O_V\big(2^{-\sqrt{s}}\big) {\rm II} \Big] O_V\Big( 2^{(3-\beta)(\sqrt{s}/100-n)}\Big),
\end{split}
\end{equation}
where ${\rm I}=\big|A(s_i) 2^{(\beta-3)s_i} -A(s'_{1-i})2^{(\beta-3)s'_{1-i}}\big|$ and ${\rm II}=A(s_i)2^{(\beta-3)s_i}+A(s'_{1-i})2^{(\beta-3)s'_{1-i}}$.

By one-point estimates \eqref{eq:sharpone} and \eqref{eq:oneballd}, we have 
$$ A (s_{i} ) \asymp \frac{2^{- (3-\beta) s_{i} }}{2^{- (3 -\beta) n}} \ \ \ \text{ and }  \ \ \  A (s_{1-i}' ) \asymp \frac{2^{- (3-\beta) s_{1-i}' }}{2^{- (3 -\beta) n}}.$$
Therefore, since $0 < \frac{3 - \beta }{100} < \frac{3}{100}$,  it follows that 
\begin{equation}\label{eq:cruxlast1}
 O_V\Big( 2^{-\sqrt{s}+(3-\beta)(\sqrt{s}/100-n)}\Big) {\rm II} \leq O_V\big( 2^{-\delta_0 \sqrt{s}}\big)
\end{equation}
for some universal $\delta_0>0$.
We now apply \eqref{eq:asymp}, then our choice of $m \geq M_2 \geq M_0$ shows that for both $i=0,1$,
\begin{equation}\label{eq:cruxlast2}
 2^{(3-\beta)(\sqrt{s}/100-n)}{\rm I}\leq C 2^{-\delta s}
\end{equation}
for some $\delta>0$. 

Finally, the claim~\eqref{eq:1vs2} follows from inserting~\eqref{eq:cruxlast1},~\eqref{eq:cruxlast2} into~\eqref{eq:cruxapplied} and then into~\eqref{eq:pingpong}.
\end{proof}

\subsection{Proof of Theorem~\ref{thm:two-point}} \label{sec:two-proof}

We are now ready to prove the existence of the two-point Green’s function.

\begin{proof}[Proof of Theorem~\ref{thm:two-point}]
By Proposition~\ref{prop:3.0}, the limit 
\begin{align}\label{eq:two-exists}
  g (z,w) 
  &\coloneqq \lim_{s \to \infty} \mathbb{E} \left( J_s (z )J_s(w) \right) \notag \\
  &= \lim_{s \to \infty} \left(2^{2(3 - \beta)s} P \left( \KK \cap B(z , 2^{-s}) \neq  \emptyset, \,   \KK \cap B(w , 2^{-s}) \neq  \emptyset \right) \right)   
\end{align} 
exists for each $z , w \in \Omega$. 
This is because, starting from~\eqref{EQ2}, it is easy to see that for $m\in\mathbb{Z}$, $J_m(V)$ has an $L^2$- and a.s.-limit (the latter requiring the use of Chebyshev inequality and Borel-Cantelli lemma), which is also the limit if we let $r$ tend to $0$ continuously. See the proof of~\cite[Theorem 3.1]{LawRaz} for a complete argument. 

Recall that $\hat{x} = (1/2, 0 , 0)$ is a reference point and that for any $z, w \in \Omega$, 
$d_z = \min \{ \vert z \vert, 1 - \vert z \vert \}$ and $ d_{z,w} = \min \{ d_z, d_w, \vert z - w\vert \} $.
  Let us fix $z, w \in \Omega$ with $ d_{z,w} > e^{-s^2} $ and take $ 0 < r < e^{-s^4}$. 
The asymptotic independence from Proposition~\ref{prop:decoup}
implies that there exists $M_2 (r,  d_{z,w}, d_z, d_w) > 0$ such that for all $m \geq M_2$
\begin{align} \label{eq:twopt-onept}
  &\frac{ 
  P \big(  \eta_m \cap B (z_m, r)  \neq \emptyset  ,   \, \eta_m \cap B (w_m, r)  \neq \emptyset  \big) 
  }{ P \big( z_m, w_m \in \eta_m \big) } \notag
  \\ \ \ \ \ & =   \left( \frac{ P \big( \eta_m \cap B(\hat{x},r)   \neq \emptyset \big) }{P \big(\hat{x} \in \eta_{m} \big)} \right)^2  \Big[ 1 + O \big(d_{z,w}^{-c} r^{\delta} \big) \Big].
\end{align}
Proposition~\ref{prop:continuity} ensures that the numerator of the LHS of \eqref{eq:twopt-onept} can be accurately approximated by  
$$ P \Big( \KK \cap B (z,r) \neq \emptyset, \KK \cap B (w,r) \neq \emptyset \Big). $$
Replacing the numerator by the hitting probability for $\KK$ above, by using Proposition~\ref{one-ball-re} and the convergence in~\eqref{eq:two-exists}, we see that the quantity 
$$  \frac{ P \big( z_m, w_m \in \eta_m \big) }{ P \big(\hat{x} \in \eta_{m} \big)^{2} } $$
converges as $m \to \infty$. Since this quantity does not depend on $r$, we can strengthen the convergence in~\eqref{eq:two-exists} as in the proof of Theorem~\ref{thm:oneballnew}, so that 
\begin{equation} \label{eq:twop-rateProof}
  P \Big( \KK \cap B (z,r) \neq \emptyset, \KK \cap B (w,r) \neq \emptyset \Big)
    =  g (z,w) r^{2(3 - \beta)} \Big[ 1 + O \big(d_{z,w}^{-c} r^{\delta}\big) \Big],
\end{equation} 
which corresponds to \eqref{eq:twop-rate}.

On the other hand, by the inclusion-exclusion principle we can write the probability of the LERW $\eta_m$ intersecting two disjoint balls as one-point hitting probabilities:
\begin{multline} \label{eq:inclexcl}
 P  \Big( \eta_m \cap B(z,r) \neq \emptyset,  \eta_m \cap B(w,s) \neq \emptyset  \Big)\\ =  
  P \Big( \eta_m \cap B(z,r) \neq \emptyset\Big) + P \Big(\eta_m \cap B(w,s) \neq \emptyset\Big) \\  - P \Big( \eta_m \cap (B(z,r) \cup B(w,r)) \neq \emptyset \Big).
\end{multline}
Similarly for $ P (\KK \cap B (z,r) \neq \emptyset, \KK \cap B(w,r) \neq \emptyset ) $.
Each one of the terms on the right-hand side of~\eqref{eq:inclexcl} approximates the intersection probability of a ball and the scaling limit of the LERW. 
Proposition~\ref{prop:Hanyscale} implies that there exist $c > 0$ and $\delta > 0$ so that for $r \in (0,   e^{-s^4})$
\begin{equation*} 
  P \left( \eta_m \cap B(z,r) \neq \emptyset \right) 
  = P \left( \KK \cap B(z,r) \neq \emptyset \right) \Big[ 1 + O \big(d^{-c}_{z,w} m^{-\delta}\big) \Big] \qquad \text{as } m \to \infty.
\end{equation*}
The corresponding estimate also holds for the other sets in the right-hand side of~\eqref{eq:inclexcl}, so that
\begin{align*}
P \Big( \eta_m \cap B(w,r) \neq \emptyset \Big) &\simeq P \Big( \KK \cap B(w,r) \neq \emptyset \Big); \\
P \Big( \eta_m \cap (B(z,r) \cup B(w,r)) \neq \emptyset \Big) &\simeq P \Big( \KK \cap (B(z,r) \cup B(w,r)) \neq \emptyset \Big).
\end{align*} 
These latter estimates, together with~\eqref{eq:inclexcl}, imply that for $c > 0$ and $\delta > 0$ and $r \in (0,   e^{-s^4})$, 
\begin{equation}\label{eq:compDC0}
    \begin{split}
        & P  \Big( \eta_m \cap B(z,r) \neq \emptyset,  \eta_m \cap B(w,r) \neq \emptyset  \Big) \\
        = \;& P  \Big( \KK \cap B(z,r) \neq \emptyset,  \KK \cap B(w,r) \neq \emptyset  \Big) \Big[ 1 + O \big(d^{-c}_{z,w} m^{-\delta}\big) \Big] \text{ as } m \to \infty.
    \end{split}
\end{equation} 
Now, \eqref{eq:compDC0} and~\eqref{eq:twop-rateProof} imply that there exist constants $c', \delta' > 0 $ so that
\begin{equation} \label{eq:compDC} 
P  \Big( \eta_m \cap B(z,r) \neq \emptyset,  \eta_m \cap B(w,r) \neq \emptyset  \Big) = 
  g (z,w) r^{2(3 - \beta)} \Big[ 1 + O \big(d_{z,w}^{-c'} r^{\delta'}\big) \Big] .
\end{equation}
Going back to~\eqref{eq:twopt-onept}, we compare its numerator on the left-hand side with~\eqref{eq:compDC}, so that for some $c'' > 0$  
\begin{equation} \label{eq:twopt-onept1}
\frac{  g(z,w) }{ P \Big( z_m, w_m \in \eta_m \Big) }
= \Big[ 1 + O \big(d_{z,w}^{-c''}r^{\delta'}\big) \Big] \left( \frac{ P \Big(\eta_m \cap B(\hat{x},r)   \neq \emptyset\Big) }{r^{3 - \beta }P \Big(\hat{x} \in \eta_{m}\Big)} \right)^2.
\end{equation}
From~\eqref{eq:twopt-onept1} and Corollary~\ref{cor:asymp}, we obtain~\eqref{eq:twopt-green}.
\end{proof}

\section{Proof of Theorem~\ref{mainthm}.}\label{se:proof}


\subsection{Existence of Minkowski content}
In this subsection, we prove part I of Theorem \ref{mainthm}, namely the following claim.
\begin{prop} \label{prop:existenceM}
For any dyadic box $V\in \cD^o$, $\cont_{\beta}(\KK\cap V)$ a.s.\ exists.
\end{prop}

To confirm the existence of the Minkowski content, we apply Proposition \ref{prop:generalMC} to $\KK$.
Consider a fixed $V\in\cD^o$ and 
recall $J_s(\cdot)$ the renormalized indicator function of the blow-up of $\KK$.  Write $$J_s(V)=\int_V J_s(z) dz$$ for the total $2^{-s}$-content of $\KK$ in $V$.

Then, it is not difficult to check that all general assumptions in the statement of Proposition \ref{prop:generalMC} are satisfied. 
Since $V\in \cD^o$, there exists an open Borel set $U \subset \DD \setminus \{ 0 \}$  such that $V \subset {\rm int} (U)$ and $\dist (0, \partial \DD, U) >0$.
Hence it suffices to prove \eqref{eq:mereexistence} and \eqref{eq:speed} for the 3D LERW, namely the following conditions for  $z , w \in U$ and  $z \neq w$. 

\begin{enumerate}
  \item \label{condition:mereexistence} The limits defining the one-point function
  \[
    g(z) 
    = \lim_{s \to \infty} E \left[ J_s (z) \right] 
    = \lim_{s \to \infty}  \left( 2^{(3-\beta)s} P \left( B(z, 2^{-s}) \cap \KK  \neq \emptyset \right) \right),
  \]
  and the two-point function 
  \[ 
    g(z,w) 
    = \lim_{s \to \infty} E \left[ J_s (z) J_s (w) \right]  
  \]
  exist and are finite. Moreover, the one-point function $g$ is uniformly bounded on $U$.
  \item \label{condition:speed} There exist $c ,  b, \rho_0, u > 0$ such that for $ 0 \leq \rho \leq \rho_0 $
  \[ 
    \begin{split}
    &g(z) \leq c; \;\; g(z,w) \leq  c \, |z-w|^{-(3-\beta)};
    \;\;   \left| E [J_s(z)] - g(z)\right| \leq c \, e^{-us^b};\\
    &\left| E [J_s(z) \, J_{s+\rho}(w)] -
    g(z,w)\right| \leq  c \, |z-w|^{-(3-\beta)} \, e^{-us^b}.
    \end{split}
  \]
\end{enumerate}

\begin{proof}[Proof of Proposition~\ref{prop:existenceM}]
  By Proposition \ref{prop:generalMC}, it suffices to verify Conditions \ref{condition:mereexistence} and \ref{condition:speed} above.

  We note that claims on one-point estimates are easy corollaries of Theorem \ref{thm:oneballnew}.
  The existence of the two-point Green's function follows from Theorem~\ref{thm:two-point} and its upper bound from~\eqref{eq:cb1}. Since we know the rate of convergence for the two-point function (see RHS of \eqref{eq:twop-rate}), the second condition is satisfied. 
\end{proof}

\subsection{Relating Minkowski content with limiting occupation measure}

In this subsection, we finish the proof of Theorem \ref{mainthm}, which will be rephrased in Proposition \ref{prop:equivprop}.

We start by relating $c_0$ in the statement of Theorem \ref{mainthm} to that of Theorem \ref{thm:oneballnew}.
Recall the definition of the reference point $\hat{x}=(1/2,0,0)$ and the mesh size $m^{-1}=2^{-n}$ for some $n\in\Zp$. For $x=(a,b,c)\in m^{-1}\Z^3$. write $$x+\square \coloneqq \left(a-\frac{m}{2},a+\frac{m}{2}\right]\times \left(b-\frac{m}{2},b+\frac{m}{2}\right]\times\left(c-\frac{m}{2},c+\frac{m}{2}\right]$$ 
for the box of size $m^{-1}$ centered at $x$.
For $s>0$, write 
$$
c_0(n,s)\coloneqq \frac{E[J^{2^n}_s(\hat{x}+\square)]}{2^{n\beta}P(\hat{x}\in \eta_{2^n})}.
$$
We now claim that $c_0(n,s)$ has a limit. Recall the definition of $g(x)$ and $c_0$ from Theorems \ref{thm:oneptnew} and \ref{thm:oneballnew} respectively.
\begin{prop}
For any $s>0$,
\begin{equation}\label{eq:cszerodef}
c_0(s)\coloneqq\lim_{n\to\infty}c_0(n,s)=2^{(3-\beta)s}\frac{P(\KK\cap B(\hat{x},2^{-s}) \neq \emptyset)}{g(\hat{x})}.
\end{equation}
Moreover,
\begin{equation}\label{eq:czerodef}
c_0=\lim_{s\to\infty}c_0(s).
\end{equation}
\end{prop}
\begin{proof}
The claim \eqref{eq:cszerodef} is an easy corollary of the definition of $J$, Proposition \ref{prop:continuity} and Theorem \ref{thm:oneptnew}, while \eqref{eq:czerodef} follows from Theorem \ref{thm:oneballnew}.
\end{proof}
\begin{prop} \label{prop:equivprop} For each $V\in \cD^o$, 
\begin{equation}\label{eq:equivprop}
J_V=c_0\mu(V)\quad\mbox{ a.s.}
\end{equation}
\end{prop}
To relate the content to the limiting occupation measure, as in last subsection, again we pass the analysis to the discrete and make use of the asymptotic independence for 3D LERW.

\begin{proof}
To prove \eqref{eq:equivprop}, it suffices to show that for each $\varepsilon>0$, $P(|J_V-c_0\mu(V)|>\varepsilon)=0$.
Suppose on the contrary there exists $C=C(\varepsilon)>0$ such that
\begin{equation}\label{eq:contra}
P\left(|J_V-c_0\mu(V)|>\varepsilon\right)\geq C.
\end{equation}
By Skorokhod's representation theorem, thanks to Theorem \ref{thm:weakconv}, we can consider $\PP$, a coupling of $(\eta_{2^n},\mu_{2^n})_{n\in\Zp}$ and $(\KK,\mu)$ such that under the coupling,
\eqref{eq:weakconv} happens $\PP$-a.s. We write $\EE$ for the expectation w.r.t.\ $\PP$.

By~\cite[Proposition 9.4]{Natural}, $V$ is a.s.\ a continuity set for $\mu$, hence we have
\begin{equation}
\mu_{2^n}(V) \to \mu(V)\quad\PP\mbox{-a.s.}
\end{equation}
As a result, there exists some $n_0(\varepsilon,C)$ such that for all $n\geq n_0$, 
\begin{equation}\label{eq:mumu}
\PP\left(\big|\mu_{2^n}(V)-\mu(V)\big|>\frac{\varepsilon}{4c_0}\right)<C/8.
\end{equation}

As $J_V$ is the a.s.\ limit of $J_s(V)$, we can pick some $s_0=s_0(\varepsilon,C)\in (0, 2^{-k^4}/100)$, such that for all $s>s_0$,
\begin{equation}\label{eq:JVJsV}
\PP\left(|J_V-J_s(V)|>\varepsilon/8\right)\leq C/8.
\end{equation}

We now claim the following proposition but postpone its proof till the end of this subsection.

\begin{prop}\label{prop:maininequ1.1}
For each $V\in \cD^o$, some $s_1(\varepsilon,C,V)>0$ such that for all $s>s_1$, there exists $n_1(s,\varepsilon,C,V)\in\Zp$ such that for 
all $n>n_1$,
\begin{equation}\label{eq:mainineq}
\PP\Big(\big|J_s^{2^n}(V) -c_0\mu_{2^n}(V) \big| \geq \varepsilon/8\Big) \leq C/16.
\end{equation}
\end{prop}
By tightness of the occupation measure (see Corollary \ref{cor:tight}), one can find some $M(C)<\infty$ such that
\begin{equation}\label{eq:mutight}
\PP(|\mu_{2^n}(V)|\geq M)\leq C/16.
\end{equation}
By \eqref{eq:mainineq}, the inequality \eqref{eq:mutight} implies that there exists some $M_1(\varepsilon,C)<\infty$, such that
\begin{equation}\label{eq:Jtight}
\PP(|J_s^{2^n}(V)|\geq  M_1)\leq C/8.
\end{equation}
We now pick $s>2\max(s_0,s_1)$ and keep it fixed. Write $\varepsilon_1=\min(\varepsilon/8M_1,1/10)$. We pick $\delta=\delta(s,\varepsilon_1)>0$ such that if we define $s',s''$ via the following equations:
$$e^{-s}+\delta=e^{-s'};\quad e^{-s}-\delta=e^{-s''},$$ 
then
$$
s''>s'>\max(s_0,s_1)\quad\mbox{ and }\quad \max \big\{ e^{(3-\beta)(s-s')} , e^{(3-\beta)(s''-s)} \big\} <1+\varepsilon_1.$$
As $\eta_{2^n}\to\KK$ a.s.\ in Hausdorff metric, one can pick some $n_2(\varepsilon,C)$ such that for all $n>n_2$,
\begin{equation}
\PP(d_{\rm Haus}(\eta_{2^n},\KK)\geq \delta)\leq C/8.
\end{equation}
This implies that 
\begin{equation}\label{eq:Hausineq}
\PP\Big(  (1-\varepsilon_1) J_{s'}(V)\leq J^{2^n}_s(V)
 \leq (1+\varepsilon_1) J_{s''}(V) \Big) \geq 1- C/8.
\end{equation}
Hence, \eqref{eq:Hausineq} and \eqref{eq:JVJsV} imply that if we pick $n>\max(n_0,n_1,n_2)$,
\begin{equation}\label{eq:Jsquick}
\PP( |J_V-  J_s^{2^n}(V)| \geq \varepsilon/4) \leq C/4.
\end{equation}

Combining \eqref{eq:mumu}, \eqref{eq:Jsquick} and  \eqref{eq:mainineq}, the  probability of the event that 
$$ | J_{V} - c_{0} \mu (V) | \le \frac{5 \varepsilon}{8} $$
is bigger than $1 - \frac{7C}{16} $, which contradicts \eqref{eq:contra}! We hence finish the proof of \eqref{eq:equivprop}.
\end{proof}

We now turn to the analysis in the discrete. 
\begin{proof}[Proof of Proposition \ref{prop:maininequ1.1}]

Write $M_2\coloneqq M_2(V)=\sup_{n}\EE[\mu^2_{2^n}(V)]$. Up-to-constant estimates for two-point functions of 3D LERW (see Proposition \ref{prop:uptocb2}) implies that $M_2<\infty$.
We fix some $k_1=k_1(k,\varepsilon,C,V)\in\Zp$ such that
\begin{equation}
2^{-k_1^2+3(k_1-k)}<\frac{(\varepsilon/16)^2}{C_2M_2C/8},
\end{equation}
where $C_2$ will be defined later (above \eqref{eq:C2def}).
Recall that 
$$
c_0(n,s)= J^{2^n}_s(\hat{x}+\square)/P(\hat{x}\in\eta_{2^n})\quad
\mbox{ and }\quad c_0=\lim_{s\to 0}\lim_{n\to\infty} c_0(n,s).$$
We now pick $s_1=s_1(\varepsilon_1,M_1)=s_1(\varepsilon,C,V)>0$ such that
$$|c_0(s)-c_0|<\frac{\varepsilon_1}{32M_1},$$ then pick $n_3=n_3(s,\varepsilon_1,M_1)=n_3(s,\varepsilon,C,V)>0$ such that for all $n>n_3$,
$$|c_0(s,n)-c_0(s)|<\frac{\varepsilon_1}{32M_1}.$$

Let $V_1,\ldots,V_{g(k_1)}\in\cD_{k_1}$ where $g(k_1)=2^{3(k_1-k)}$ be the collection of dyadic boxes of scale $k_1$ that partition $V$.

For $s>s_1$ and $n>n_3$, we 
apply Proposition~\ref{prop:decoup} with $V_i$ for each $i=1,\ldots,g(k_1)$ similarly as how we apply~\cite[Proposition 6.2]{Natural} to prove~\cite[Proposition 6.1]{Natural}, and obtain that
there exists some universal $C_2>0$ such that for all $n>k_1^4)$,
\begin{equation}\label{eq:C2def}
\EE\Big[\big(J^{2^n}_s(V_i)-c_0(n,s)\mu_{2^n}(V_i)\big)^2\Big]\leq C_22^{-k_1^2}\EE[\mu^2_{2^n}(V_i)].
\end{equation}

Note that the $C_2$ does not rely on $k$.
By adding them together and applying Cauchy-Schwarz inequality, we obtain
\begin{equation}\label{eq:blah}
\EE\Big[\big(J^{2^n}_s(V)-c_0(n,s)\mu_{2^n}(V)\big)^2\Big]\leq C_22^{-k_1^2+3(k_1-k)}\EE[\mu^2_{2^n}(V)].
\end{equation}
We now take $n_1=n_1(s,\varepsilon,C,V)\coloneqq\max(n_3,k_1^4)$. By our choice of $k_1$, the RHS of \eqref{eq:blah} is bounded by $\frac{\varepsilon^2/16}{C/8}$. Hence, by Markov inequality, this implies
\begin{equation}\label{eq:onemorestep}
\PP\Big(\big|J_s^{2^n}(V) -c_0(n,s)\mu_{2^n}(V) \big| \geq \varepsilon/16\Big) \leq C/8.
\end{equation}
As we have picked $n,s$ such that $|c_0(n,s)-c_0|\leq \varepsilon/(16M_1)$, by \eqref{eq:mutight},
\begin{align*}
&\PP\Big(\big|J_s^{2^n}(V) -c_0\mu_{2^n}(V) \big| \geq \varepsilon/8\Big) \\ &\leq \PP\Big(\big|J_s^{2^n}(V) -c_0(n,s)\mu_{2^n}(V) \big| \geq \varepsilon/16\Big)+\PP[\mu_{2^n}(V)>M_1]\leq C/4.
\end{align*}
This finishes the proof of \eqref{eq:mainineq}.
\end{proof}
We have thus finished the proof for Theorem \ref{mainthm}. 

\subsection{Invariance}

As Minkowski content is a scaling-invariant and rotation-invariant quantity, Theorem \ref{mainthm} implies that $\eta$, the scaling limit of 3D LERW in natural parametrization, is also scaling-invariant and rotation-invariant in the following sense.
\begin{cor}\label{cor:maincor}
For any $a>0$, let $\KK^a$ be the scaling limit\footnote{Its existence has also been  confirmed in \cite{Kozma}.} of 3D LERW of mesh size $m^{-1}$, $m\in\Zp$ in $a\mathbb{D}$. 
Write $\eta^a$ for the parametrization of $\KK^a$ by its Minkowski content. 
Then it follows that
\[ 
  \eta^a(\cdot) \overset{\rm d}{=} a \eta(a^{\beta} \cdot).
\]
For any $\theta\in[0,2\pi)$, let $T_\theta(\eta)$ stand for $\eta$ rotated (counter-clockwise) by angle $\theta$, then
\[
  T_\theta(\eta)\overset{\rm d}{=}\eta.
\]
\end{cor}

\begin{proof}
  As discussed above in~\eqref{eq:scaleInvariance}, Subsection 6.1 in \cite{Kozma} shows that $\KK$ is both scale-invariant and rotation-invariant. 
  Then $ \KK^a \overset{\rm d}{=} a \KK $  and 
  $ T_\theta(\KK)\overset{\rm d}{=} \KK $.
  The the corresponding invariance of the Minkowski content and the definition of the parametrization by the Minkowski content in~\eqref{eq:paramMinkowski} imply the conclusion of the corollary.
\end{proof}

\section*{Acknowledgments}
The authors would like to thank the anonymous referees for very detailed and helpful comments.

SHT would like to acknowledge the support of UNAM-PAPIIT grant IA103724. XL is supported by National Key R\&D Program of China (No.~2021YFA1002700 and No.~2020YFA0712900) and NSFC (No.~12071012). DS was supported by JSPS Grant-in-Aid for Scientific Research (C) 22K03336, JSPS Grant-in-Aid for Scientific Research (B) 22H01128 and 21H00989.

\end{document}